\def\draft{n}
\documentclass{amsart} 
\usepackage{latexsym,amssymb,graphicx,amscd,amsmath,url}

\def\printname#1{
	\if\draft y
		\smash{\makebox[0pt]{\hspace{-0.5in}
			\raisebox{8pt}{\tt\tiny #1}}}
	\fi
}

\def\lbl#1{\label{#1}\printname{#1}}

\newtheorem{thm}{Theorem}[section]
\newtheorem{lem}[thm]{Lemma}
\newtheorem{prop}[thm]{Proposition}
\newtheorem{cor}[thm]{Corollary} 
\newtheorem{de}[thm]{Definition} 
\newtheorem{rem}[thm]{Remark}

\newtheorem{conv}[thm]{Convention}
\newtheorem{notation}[thm]{Notation}
\newtheorem{rems}[thm]{Remarks}

\newcommand{\BSp}{{\mathcal{S}_p}}
\newcommand{\BS}{{\mathcal{S}}}
\newcommand{\BSplus}{{\mathcal{S}^+_p}}

\newcommand{\KK}{{\mathcal{K}}}
\newcommand{\hKK}{{\hat{\mathcal{K}}}}

\newcommand{\BZ}{{\mathbb{Z}}}
\newcommand{\BQ}{{\mathbb{Q}}}
\newcommand{\BO}{{\mathcal{O}}}

\newcommand{\BH}{{\mathcal{H}}}

\newcommand{\BF}{{\mathbb{F}}}

\newcommand{\Si}{{\Sigma}}

\newcommand{\cE}{{\mathcal{E}}}

\newcommand{\bb}{{\mathfrak{b}}}
\newcommand{\bt}{\tilde{\mathfrak{b}}}

\newcommand{\be}{{\beta}}
\newcommand{\up}{{\eta}}

\newcommand{\fo}{\mathfrak{o}}
\newcommand{\fe}{\mathfrak{e}}

\newcommand{\I}{{\mathrm I}}

\DeclareMathOperator{\varc}{C}
\DeclareMathOperator{\Image}{Image}

\DeclareMathOperator{\End}{End}
\DeclareMathOperator{\Hom}{Hom}
\DeclareMathOperator{\Id}{Id}
\DeclareMathOperator{\Sp}{Sp}
\DeclareMathOperator{\rank}{rank}
\DeclareMathOperator{\SL}{SL}

\DeclareMathOperator{\SU}{SU}
\DeclareMathOperator{\SO}{SO}
\DeclareMathOperator{\diag}{diag}
\DeclareMathOperator{\res}{res}

\begin{document}

\title[Irreducible factors   
of modular representations in TQFT] 
{Irreducible factors 
of modular representations of mapping
  class groups arising in Integral
TQFT}
 
\author{ Patrick M. Gilmer}
\address{Department of Mathematics\\
Louisiana State University\\
Baton Rouge, LA 70803\\
USA}
\email{gilmer@math.lsu.edu}
\thanks{The first author was partially supported by  NSF-DMS-0905736 }
\urladdr{\url{www.math.lsu.edu/~gilmer}}

\author{Gregor Masbaum}
\address{Institut de Math{\'e}matiques de Jussieu (UMR 7586 du CNRS)\\
Case 247, 4 pl. Jussieu,
75252 Paris Cedex 5\\
FRANCE }
\email{masbaum@math.jussieu.fr}
\urladdr{\url{www.math.jussieu.fr/~masbaum}}

\date{October 30, 2012}


\begin{abstract} We find  decomposition series of length at most two
  for 
modular representations in 
positive  
characteristic 
of mapping class groups of 
surfaces induced by an integral version of the Witten-Reshetikhin-Turaev $\SO(3)$-TQFT
 at the $p$-th root of unity, where $p$ is an odd prime.    
The dimensions of the irreducible factors are given by 
Verlinde-type 
formulas.
\end{abstract}

\maketitle
\tableofcontents

\section{Introduction} 

The  
Witten-Reshetikhin-Turaev quantum invariants of
$3$-mani\-folds fit into a Topological Quantum Field Theory (TQFT) in the sense of Atiyah and
Segal. This means in particular that they give rise to
finite-dimensional complex 
representations of (certain 
central extensions of)  mapping
class groups of surfaces.  While these {\em quantum representations} of
mapping class groups are in general not
irreducible (some decompositions into invariant subspaces were already
described in  \cite{BHMV2}), it was shown by Roberts \cite{R} that in
the case of surfaces without boundary, 
the representations arising from the $\SU(2)$-theory at level $k$ are
irreducible provided $k=p-2$ where $p$ is a prime. This irreducibility
plays an important role in Andersen's 
work on whether mapping class groups have Kazhdan's property T
\cite{A}.

It is a well-known phenomenon in representation theory that `natural'
representations of a group can often be defined both in
characteristic zero and in positive characteristic. 
 Typically, though, a representation which is irreducible
in characteristic zero 
gives rise  
in positive characteristic to a
representation which may no longer be
irreducible. It is the purpose of this paper to study this question
for  
mapping class group representations coming from a 
TQFT. Specifically, we consider the 
$\SO(3)$-TQFT 
at a primitive $p$-th root of unity, 
where $p\geq 5$
is a prime, because the theory of {\em Integral
TQFT} developed in \cite{G, GM, GM2} implies that 
this TQFT 
gives
rise in a natural way to modular representations of mapping class
groups in positive characteristic.  
We remark that modular representations in
 characteristic different from $p$  
coming from this $\SO(3)$-TQFT at the $p$-th root of unity 
were recently used in work of A. Reid and the second
author \cite{MRe}. But here  we will 
mainly  
study the modular representations in characteristic $p$ coming from
this  theory. 
We will show that these representations  often have a nontrivial
composition series with interesting irreducible factors. 

\section{Statement of the main results}
 \lbl{sec.intro}

Throughout the paper, we fix a prime $p\geq 5$, and let $\zeta_p$ be a primitive  $p$-th root
 of unity.  
We denote the corresponding ring of cyclotomic integers by
$$\BO=\BZ[\zeta_p]~.$$   
The 
theory of Integral $\SO(3)$-TQFT  
developed in \cite{G, GM, GM2} 
associates to a
compact oriented surface $\Si$ a free $\BO$-lattice ({\em i.e.,} 
 a free $\BO$-module of finite rank)
carrying a representation  
of an appropriate central extension
of the orientation-preserving mapping class group of the
surface. This theory may be thought of as an
integral refinement of the Reshetikhin-Turaev TQFT
associated with the Lie group $\SO(3)$, a version of which would be
obtained if one extends coefficients from $\BO$ to the cyclotomic
field $\BQ(\zeta_p)$. 
In this paper, we 
write
$\BS(\Si)$ for the Integral TQFT-lattice 
associated 
to $\Si$. 
 We shall review the definition of $\BS(\Si)$ in \S\ref{sec3}.  
We write
$\Gamma_{\Si} $ for the mapping class group of the surface, and
$\widetilde \Gamma_\Si^{++}$ for the central extension of
$\Gamma_{\Si} $ which 
we will use. Here we follow the construction of $\widetilde
\Gamma_\Si^{++}$ in \cite{GM3}. The group $\widetilde
\Gamma_\Si^{++}$ is a central extension of
$\Gamma_{\Si} $ by $\BZ$. It is generated by certain specific lifts
of Dehn twists, and we will describe in Proposition~\ref{act} 
how such a lift of a Dehn twist acts on $\BS(\Si)$. It will not be
necessary to know more about 
 $\widetilde
\Gamma_\Si^{++}$
in this paper.

For every ideal $I\subset \BO$, we have an induced representation  of
$\widetilde \Gamma_\Si^{++}$ on the  
      \hbox{$\BO\slash I$-}module 
$\BS(\Si)\slash I  
\BS(\Si)$. In this way, we can get modular representations of
$\widetilde \Gamma_\Si^{++}$, namely when   
      $\BO\slash I$ 
is a finite
field. We will see in Corollary~\ref{irred}  that in the
case where $\Si$ has at most one boundary
component,
such a modular representation, say in characteristic $\ell$, is always
irreducible except possibly if $\ell=p$. In this paper, however, we are
interested in non-trivial decompositions of the
representation. Therefore we consider the case where 
$I=(1-\zeta_p)$. It is well-known that 
this is a prime ideal in $\BO=\BZ[\zeta_p]$, and the quotient ring 
is the finite
field $\BF_p$. Thus    
we get 
a
representation  of $\widetilde \Gamma_\Si^{++}$ 
on the $\BF_p$-vector space 
\begin{equation}\lbl{Fdef}
F(\Si)
= \BS(\Si)\slash 
(1-\zeta_p)  
\BS(\Si)~.
\end{equation}
 This is the modular representation referred to in the
title of this paper. As shown  
in \cite[\S12]{GM3}, it factors
through a representation of the ordinary mapping class group
$\Gamma_\Si$. 
(We shall explain the reason for this in \S\ref{sec3}.\footnote{See
  footnote~\ref{foot9}.})     
In Corollary~\ref{maincor}, 
we will show that   
in the
case where
$\Si$ is either closed, or has one boundary
component, the modular representation $F(\Si)$
has a composition series with at most
two irreducible factors, which we will describe explicitly and whose
dimensions we will compute.

In order to state 
our results
more precisely, 
recall that if $\Si$ has
boundary, the quantum  representation 
depends
on the choice of a {\em color} on each boundary component, and we get
a 
representation of the mapping
class group of $\Si$ {\em rel.~boundary} (meaning that $\Gamma_\Si$
consists of orientation
preserving diffeomorphisms which are the identity on the boundary,
modulo isotopies which are also the identity on the boundary).
 In our
situation, the color on the boundary 
can be any even integer $2c$  satisfying $0\leq
2c\leq p-3$. We will indicate this in our notation by  writing $\Si_g(2c)$ for a genus $g$ surface with one
boundary component 
colored $2c$. As usual in TQFT, the case $2c=0$
corresponds  to the case where
$\Si$ has empty boundary. 
\begin{conv}
Throughout most of the paper, $g$ and $c$ are fixed, and we simply
write $\Si$ for 
$\Si_g(2c)$. 
\end{conv} 

The lattice $\BS(\Si)$ has a basis 
$\{\bb_\sigma\}$  
indexed by small admissible
colorings $\sigma$ of the edges of 
a lollipop tree \cite{GM}. This is pictured  in  Figure~\ref{lol} in the case
$g=3$.

\begin{figure}[h]
\includegraphics[width=1.8in]{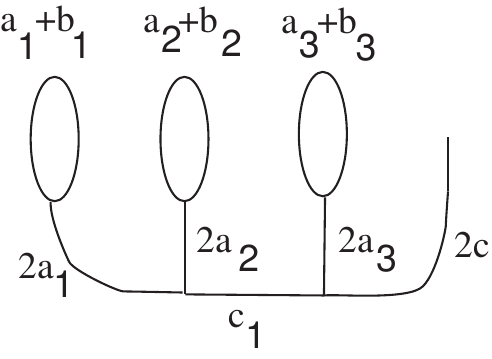}
\caption{Lollipop tree 
$T_g^{(2c)}$
}\lbl{lol}
\end{figure}

Here, the color $2c$ on the free (or {\em trunk}) edge is fixed, whereas the 
colors $2a_i, a_i+b_i, $ and $c_i$ on the other edges are numbers in $\{0,1, \ldots,
p-2\}$.  A coloring is 
{\em admissible}   
if whenever $i$, $j$ and $k$ are the colors of edges which meet at a vertex
\begin{eqnarray*}\label{adm1} i+j+k &\equiv& 0 \pmod 2\\
\label{adm2} |i-j|\ \leq& k &  \leq \ i+j\\
\label{adm3} i+j+k &\leq& 2p-4 \end{eqnarray*}
A coloring is 
{\em small}  
if the colors $a_i+b_i$ of the loop edges satisfy $0\leq a_i+b_i\leq
(p-3)/2$ 
(see \cite[\S3]{GM}).

\begin{conv} From now on, whenever we say
  coloring, we mean small admissible coloring.
\end{conv}

Here is a crucial definition for this paper.

\begin{de}\lbl{eoc} A coloring $\sigma$ is called  
even or odd
according as $c+\sum a_i$
  is  even or odd.
  \footnote{In the special case that 
$(g,c)=(2,0)$, we take $c+\sum a_i$ to be $2a_1$.}
 \end{de}

We write 
\begin{equation}\lbl{decomp}
\BS(\Si)=\BS^{ev}(\Si) \oplus \BS^{odd}(\Si)~,
\end{equation}
 where 
$\BS^{ev}(\Si)$ is the 
submodule
spanned by the $\bb_\sigma$
corresponding to even colorings, and $\BS^{odd}(\Si)$ is the  
submodule
spanned by the $\bb_\sigma$
corresponding to odd colorings.  
Note that while $\BS(\Si)$ is defined intrinsically and does not depend
on the choice of the lollipop tree, the submodules $\BS^{ev}(\Si)$ and
$\BS^{odd}(\Si)$ depend on this choice.

The following is the key result of this paper.

\begin{thm} 
\lbl{mainth} 
Let $\Si=\Si_g(2c)$. 
With respect to the decomposition
  (\ref{decomp}), the image of the group algebra $\BO[\widetilde
  \Gamma_\Si^{++}]$ in $\End_\BO(\BS(\Si))$ is  
\begin{equation}\lbl{MT}
\left[
\begin{array}{cc}
\End_\BO(\BS^{ev}(\Si))  &\multicolumn{1}{|c}{ 
(1-\zeta_p) 
\Hom_\BO(\BS^{odd}(\Si), \BS^{ev}(\Si))} \\ \hline
\Hom_\BO(\BS^{ev}(\Si), \BS^{odd}(\Si)) & \multicolumn{1}{|c}{ \End_\BO(\BS^{odd}(\Si))}
\end{array}\right] 
\end{equation} 
\end{thm}
The proof of Theorem~\ref{mainth} will be given in 
\S\ref{sec3}--\ref{sec5}.
It has the following immediate corollary for the modular
representation $F(\Si)$  
of $\Gamma_\Si$ 
defined in (\ref{Fdef}). 
\begin{cor}\lbl{maincor1} Let \begin{equation*}\lbl{decomp2}
F(\Si)=F^{ev}(\Si) \oplus
F^{odd}(\Si)
\end{equation*}  be the decomposition induced from
(\ref{decomp}). With respect to this decomposition, the image of the group algebra $\BF_p[\Gamma_\Si]$
  in $\End_{\BF_p}(F(\Si))$ is
\begin{equation} \lbl{decomp3}
\left[
\begin{array}{cc}
\End_{\BF_p}(F^{ev}(\Si))  &\multicolumn{1}{|c}{0} \\ \hline
\Hom_{\BF_p}(F^{ev}(\Si), F^{odd}(\Si)) & \multicolumn{1}{|c}{ \End_{\BF_p}(F^{odd}(\Si))}
\end{array}\right]
\end{equation}
\end{cor}

{}From (\ref{decomp3}) it is 
easy to get a composition series for $F(\Si)$:

\begin{cor}\lbl{maincor} Let $\Si=\Si_g(2c)$.  If $g\leq 1$ or if
  $(g,c)=(2,0)$, then  $F^{odd}(\Si)=0$ and the representation
 $F(\Si)$ 
is
irreducible. 
Otherwise, $F(\Si)$  has a composition series with two
irreducible factors:  the subspace $F^{odd}(\Si)$ is the unique irreducible
subrepresentation of $F(\Si)$, and 
the quotient
$F(\Si)\slash F^{odd}(\Si)$ is again irreducible. 
\end{cor}
We will explain how one deduces this corollary from (\ref{decomp3}) 
in \S\ref{sec2}.

Let us denote the dimensions of  these irreducible factors
  by
\begin{align*}
\fo_g^{(2c)} &= \dim_{\BF_p}(F^{odd}(\Si))=  \rank_\BO(\BS^{odd}(\Si))\\
\fe_g^{(2c)} &= \dim_{\BF_p}(F(\Si)\slash F^{odd}(\Si))=  \rank_\BO(\BS^{ev}(\Si))
\end{align*}
where $\Si=\Si_g(2c)$. 
These numbers are simply the  
numbers
of odd or even colorings of the lollipop
tree  
$T_g^{(2c)}$ 
shown in Figure \ref{lol}. Here, the superscript indicates that the
trunk color is fixed to be $2c$ in the colorings we are counting. We
also define 
$$D_g^{(2c)}= \fe_g^{(2c)} +
\fo_g^{(2c)}, \ \ \ \delta_g^{(2c)}= \fe_g^{(2c)} -
\fo_g^{(2c)}~.$$  In the case $c=0$ 
corresponding  to a surface without boundary, we simply
write $\fe_g$ for $\fe_g^{(0)}$, and similarly for the other numbers just
defined. 

We will give recursion formulas for $\fo_g^{(2c)}$ and $
\fe_g^{(2c)}$ in \S\ref{dim}, where we will also prove the following
explicit expression for $\delta_g^{(2c)}$:

\begin{thm}\lbl{new-Verl} One has 
\begin{equation}\lbl{d-Verl} (-1)^c \delta_g^{(2c)}=  \frac{4^{1-g}}{p}
  \sum_{j=1}^{(p-1)/2}  \left(\sin\frac{\pi j (2c+1)}{p}\right) \left(\sin\frac{\pi j}{p}\right) \left(\cos\frac{\pi j}{p} \right)^{-2g}.
\end{equation} 
\end{thm}

We think of (\ref{d-Verl}) as an analog of the celebrated Verlinde 
formula for the dimension of the $\SO(3)$-TQFT vector space
$V_p(\Si)$ (in the notation of
\cite{BHMV2}).  In our current notation, this dimension is
$$\dim V_p(\Si)= 
\rank_\BO(\BS(\Si))= \dim_{\BF_p}(F(\Si))=
 \fo_g^{(2c)}+
\fe_g^{(2c)}= 
D_g^{(2c)}, $$ and the Verlinde formula (for $p$ odd)
 says\footnote{The formula for $D_g$ obtained by substituting $c=0$ in
  (\ref{Verl}) is different from the formula for $d_g(p)$ in
  \cite[p.~896]{BHMV2}, but they 
have  the same sum when $p$ is odd.} 
\begin{equation}\lbl{Verl} 
D_g^{(2c)}=\left( \frac p 4\right)^{g-1} \sum_{j=1}^{(p-1)/2}
\left(\sin \frac { \pi j(2c+1)}{p}\right) \left( \sin \frac { \pi j}{p} \right)^{1-2g}
.\end{equation} 
One can, of course,  combine (\ref{d-Verl}) and (\ref{Verl}) to get similar expressions for the numbers $\fe_g^{(2c)}$ and
$\fo_g^{(2c)}$.  
Note 
also   
that if one substitutes
 $p/\sin^{2}(\pi j/p)$ for $1/\cos^{2}(\pi j/p)$
in the right hand side of our formula (\ref{d-Verl}), one gets the right hand side of the Verlinde formula
(\ref{Verl}). This will be explained in the proof of 
 Theorem~\ref{new-Verl} in \S\ref{dim.2}.

 It is well-known  that for $g\geq 2$, the number $D_g$ 
when viewed as a function of $p$ is a
polynomial of degree $3g-3$ in $p$ 
      (see for example  Zagier \cite{Z} or \cite{BHMV2}).
Moreover, the number  $D_g^{(2c)}$ is a  polynomial
in $p$ and $c$ of total degree $3g-2$. 
(This follows from a residue formula for $D_g^{(2c)}$ which we shall
give in 
\S\ref{dim.4}.) 
  We have a similar result  for $\delta_g^{(2c)}$.  To state
it, let $B_n$ be the Bernoulli numbers defined by 
\begin{equation}\lbl{bernn}
\frac t {e^t-1}= \sum_{n=0}^\infty B_n
\frac {t^n} {n!}~.
\end{equation}
 Note $B_1=-1/2$ and $B_k=0$ for all odd $k\geq
3$. It will be convenient to denote (see \cite[p.~231]{IR})
$$\widetilde B_n=\frac 1 2 \frac{(-1)^{n+1} B_{2n}}{(2n)!}=
\frac{\zeta(2n)}{(2\pi)^{2n}}~.$$

 \begin{thm}\lbl{deltaprop} For 
 $g\geq 1$,
 we have that 
  $\delta_g^{(2c)}$
  is an (inhomogeneous) polynomial in $c$ and $p$ of total degree
  $2g-1$. The 
homogeneous part of degree $2g-1$
in this polynomial is given by
\begin{equation}\lbl{LG} 
(-1)^{g-1} \sum_{k=1}^{2g} \,2(2^k-1)\, \frac{B_k}{k!}\,
\frac{c^{2g-k}}{(2g-k)!}\, p^{k-1}~.
\end{equation}
 In particular, for any fixed value of $c\in \{0,1,\ldots,(p-3)/2\}$,  
the polynomial  $\delta_g^{(2c)}(p)$ has degree $2g-1$ in $p$.
Moreover, the leading coefficient of this polynomial 
is \hbox{$4(2^{2g}-1)\widetilde B_g$.}
\end{thm}

This should be compared to the fact that the leading coefficient of
the polynomial $D_g(p)$ is $\widetilde B_{g-1}$ (see Formula (6) in Zagier
\cite{Z}\footnote{Our 
$D_g$ is $2^{-g}$ times
Zagier's $D(g,p)$, which is the dimension of the $\SU(2)$-TQFT vector space
 $V_{2p}(\Si_g)$. This fact comes from a tensor product formula, see  
 \cite[Theorem~1.5]{BHMV2}.  
The  $V_{2p}$-theory
         corresponds in Conformal Field Theory to $\SU(2)$ at level
         $k=p-2$.} 
or Corollary~\ref{1.10} below).   
 This leading coefficient is closely
related to the volume of a certain moduli space \cite[\S3]{Wi}.  
Moreover, the dimension of this moduli space is equal to the degree of the
polynomial $D_g(p)$, by the Riemann-Roch formula. 
We wonder whether there is an algebro-geometric interpretation of the
degree and the leading coefficient of the 
polynomial $\delta_g(p)$. We remark that in the related but different
situation of $\SU(2)$-TQFT at even level, a similar question about an
algebro-geometric 
interpretation of the difference of the number of even and odd
colorings (albeit for a quite different notion of even and odd colorings)  was
answered affirmatively in \cite{AM} in terms of certain geometrically
defined involutions on
the moduli space.

Theorem~\ref{deltaprop} will be proved in \S\ref{dim.3}. It implies that the dimensions  $\fo_g^{(2c)}$
and $\fe_g^{(2c)}$   of our irreducible factors  $F^{odd}(\Si)$ and
$F(\Si)\slash F^{odd}(\Si)$ are polynomials in $p$ and $c$ which
coincide with the polynomial $D_g^{(2c)}/2$ in degrees higher than
$2g-1$. This leads to the following corollary, which will be proved in \S\ref{dim.4}.

\begin{cor}\lbl{1.9} For $g\geq 2$, both $\fe_g^{(2c)}$ and
  $\fo_g^{(2c)}$ are polynomials of total degree $3g-2$ in $p$ and
  $c$.  For $g\geq 3$, the leading terms of these polynomials  are given
  by 
\begin{equation}\lbl{1.99} \fe_g^{(2c)} \equiv \fo_g^{(2c)}\equiv
  \frac{(-1)^g}{4} \sum_{k=0}^{2g-2} \frac{B_k}{k!}\frac
  {c^{2g-2-k}}{(2g-2-k)!} \left(1+\frac{2c}{2g-1-k}\right)p^{g-1+k}
\end{equation}
where $\equiv$ means equality up to addition of a polynomial in $p$
and $c$ of total degree $\leq 3g-4$.
\end{cor}

Knowing the terms of total degree $3g-2$ and $ 3g-3$ in $\fe_g^{(2c)}$ and
  $\fo_g^{(2c)}$ (as given in the preceding Corollary~\ref{1.9})
  allows one to compute the leading
  term of $\fe_g^{(2c)}$ and
  $\fo_g^{(2c)}$ when viewed as polynomials in $p$, with $c$ held 
  fixed. Corollary~\ref{1.9} implies that these leading terms are
  given as follows: 

\begin{cor}\lbl{1.10} Assume $c\in
  \{0,1,\ldots,(p-3)/2\}$ is fixed. If $g\geq 3$, then both $\fe_g^{(2c)}$ and
  $\fo_g^{(2c)}$ are polynomials of degree $3g-3$ in $p$ with leading
  coefficient $(c+\frac 1 2)\widetilde B_{g-1}$. 
\end{cor}

For $g=2$, we will write down $\fe_2^{(2c)}$ and
  $\fo_2^{(2c)}$ in Eqns. (\ref{fex}) and (\ref{fox}) in 
  \S\ref{dim.1}.  In this case, the leading terms  are slightly different: for fixed
$c$, one has that $\fe_2^{(2c)}$ and
  $\fo_2^{(2c)}$ are polynomials of degree $3$ in $p$, with leading
  coefficient 
  $(c+1)/24$ and $c/24$,
  respectively, except for
  $\fo_2=\fo_2^{(0)}$ which is identically zero. 

In \S\ref{7.5}, we will give a residue  formula for $\delta_g^{(2c)}$
similar to the residue formula for $D_g^{(2c)}$ mentioned above.

\begin{rems}{\em (1) One can show \cite{M}  
that the modular representations 
$F^{odd}(\Si)$ and
$F(\Si)\slash F^{odd}(\Si)$ of the mapping class group 
  $\Gamma_\Si$ 
factor through the natural map  
 $$\Gamma_\Si \twoheadrightarrow \Sp(2g,\BZ)\twoheadrightarrow \Sp(2g,\BF_p)~.$$
 (But this is in general not the case for the representation 
$F(\Si)$ itself.)
We wonder about how to characterize our irreducible factors  among
the modular
representations 
of 
the symplectic group 
$\Sp(2g,\BF_p)$ in characteristic $p$. 
As a step in this direction,  in the case $p=5$, we have identified our dimensions $\fe_g^{(2c)}(5)$ and $\fo_g^{(2c)}(5)$  
with the dimensions of some known modular representations
 (see \S\ref{sec7}).

 (2) In the genus one case, where $F(\Si)$ itself is irreducible, we have $\fe_1^{(2c)}(p)= (p-1)/2-c$, and
 $\fo_1^{(2c)}(p)=0$. In this case, we   
computed the modular representation $F(\Sigma_1(2c))$ for any $p$
and $c$ already in \cite{GM2}, where we showed that $F(\Sigma_1(2c))$ (which was denoted by
${\mathcal{S}}^+_{p,0}({\mathcal T}_c)$ in \cite{GM2}) is isomorphic to the space of homogeneous
polynomials over $\BF_p$ in two variables of total degree
$(p-3)/2-c$. It is
well-known that this representation
of $\Sp(2,\BF_p)=\SL(2,\BF_p)$ is irreducible for every $c\leq
(p-3)/2$.  

(3) Theorem~\ref{mainth} also gives information on the representations
of the extended 
mapping class group on the quotients 
$$\BS(\Si)\slash  (1-\zeta_p)^{N}  \BS(\Si)$$ for any $N\geq 1$.  
The study of these 
higher 
quotient 
representations is interesting, because they 
factor through finite quotients of the extended mapping class group, 
but
approximate the  TQFT-representation better and better as $N$
increases. This is discussed in more detail in \cite{GM2}.

(4)  
The proper
 irreducible subrepresentation 
$F^{odd}(\Si)$ of $F(\Si)$ 
      (in the case $g\geq 2$ and  $(g,c) \ne (2,0)$) 
     is the radical of a certain $\Gamma_\Si$-invariant  
bilinear
form on  the $\BF_p$-vector space  $F(\Si)$. 
 This form is symmetric if 
       $c\equiv  ((p+1)/2) \, g\pmod 2$.
Otherwise this form is
skew-symmetric. It is induced
(after rescaling by 
$(1-\zeta_p)^{-c}$) 
 from 
 a
  natural
    $\BZ[\zeta_p]$-valued form on 
the lattice
    $\BS(\Si)$.  See \cite[\S14]{GM} for an explicit
    description of this form in  the
    case $c=0$ ({\em i.e.} when the surface has 
empty
boundary). 
This    description generalizes to the case $c\neq 0$.~\footnote
{If $c\neq 0$, one should define   
the   notion of even and odd colorings as in Definition~\ref{eoc}. The
      definition of even and odd colorings in \cite[p.~838]{GM} applies
      only to the case $c=0$.}
      }
\end{rems}

\section{Proof of Corollary~\ref{maincor} and irreducibility  
in characteristic $\neq p$}\lbl{sec2}

This
is based
on the
following well-known lemma, whose proof we omit.

\begin{lem}\lbl{onto} Suppose a group $G$  
is
represented on a vector space $V$ over a field $k$. If the associated algebra morphism $k[G] \rightarrow \End_k(V)$ is surjective, then the representation  
is irreducible.
\end{lem}

\begin{proof}[Proof of Corollary~\ref{maincor} assuming
  Theorem~\ref{mainth}]

If $g\leq 1$ or if $(g,c)=(2,0)$, all colorings are even and hence
$F^{odd}(\Si)=0$. Thus, 
by (\ref{decomp3}) (which follows immediately from
Theorem~\ref{mainth}),  
the image of $\BF_p[\Gamma_\Si]$ is 
$\End_{\BF_p}(F(\Si))$ and Lemma \ref{onto} implies that the representation
is irreducible in this case. Otherwise, both $F^{odd}(\Si)$ and
$F^{ev}(\Si)$ are non-zero. Applying again Lemma \ref{onto}, it follows from (\ref{decomp3}) that
$F^{odd}(\Si)$ is an
irreducible subrepresentation of $F(\Si)$, and the quotient
representation $F(\Si)\slash F^{odd}(\Si)$ with the induced action is again irreducible. 

We only need to show, in the case that $F^{odd}(\Si) \ne 0$, that $F^{odd}(\Si)$ is the only
proper non-trivial subrepresentation of $F(\Si)$. 
The argument is quite standard, but we give details for the reader's
convenience. Suppose, then, that  $V \subset F(\Si)$ is a proper non-trivial subrepresentation. 

We consider first the case that $V \cap F^{odd}(\Si) \ne 0$. Then $V
\supset F^{odd}(\Si)$ since 
$F^{odd}(\Si)$ is irreducible.
If 
$V$ is strictly bigger than $F^{odd}(\Si)$, 
 then $\bar V$, the image of 
$V$ in $F(\Si)/ F^{odd}(\Si)$, is non-zero.  
But $F(\Si)/ F^{odd}(\Si)$ is also irreducible.
So $\bar V = F(\Si)/ F^{odd}(\Si)$
and hence $V=F(\Si)$ which is a contradiction. Thus $V=F^{odd}(\Si)$
in this first case.

We consider now the 
second case,  where
$V \cap F^{odd}(\Si) = 0$. 
We will see that this case cannot happen. We proceed as follows. If $V
\cap F^{odd}(\Si) = 0$, then
$\bar
V$, the image of $V$ in $F(\Si)/ F^{odd}(\Si)$ is non-zero, 
hence equal to $F(\Si)/ F^{odd}(\Si)$, by irreducibility.  
It follows that $$\dim V =\dim \bar V =\dim(F(\Si)/ F^{odd}(\Si))=
\dim  F^{ev}(\Si)$$
and the projection $\pi= \Id_{F^{ev}(\Si)}\oplus\, 0_{F^{odd}(\Si)}$ from $F(\Si)$ onto $F^{ev}(\Si)$ along
$F^{odd}(\Si)$ sends $V$ isomorphically to $F^{ev}(\Si)$. 

Now by (\ref{decomp3}) the projection $\pi$ lies in the image of
$\BF_p[\Gamma_\Si]$. 
As $V$ is stable under the action of $\Gamma_\Si$, it follows that
$V$ must be equal to $F^{ev}(\Si) .$ But $F^{ev}(\Si)$ is not stable under the action,
since 
$\Image(\BF_p[\Gamma_\Si]) \supset \Hom_{\BF_p}(F^{ev}(\Si),
F^{odd}(\Si))$. This contradiction shows that the second case does not
occur.  
\end{proof}

Theorem~\ref{mainth} also implies the following
corollary for the representation over the complex numbers
which generalizes a result of Roberts \cite{R} (see Remark~\ref{Robrem}).

\begin{cor}\lbl{ci} If the surface $\Si$ has at most
  one boundary component and if $p$ is a prime, the representation  of
  the extended mapping class group 
on
  the $\SO(3)$-TQFT vector space at the $p$-th root of unity is
  irreducible  over the complex numbers. 
\end{cor}

Note that irreducibility over the complex numbers implies irreducibility
over the cyclotomic field  $\BQ(\zeta_p)$.

\begin{rem} \lbl{Robrem}
{ \em  
In the case where $\Si$ is closed, Corollary~\ref{ci} is due to Roberts \cite{R}. Roberts
actually considered the $\SU(2)$-TQFT-representation.   
But when the representation is considered over the
     complex numbers, his argument works
     also in the $\SO(3)$-case we
     are considering in this paper. Roberts' argument is, however,
     quite different from ours.
}\end{rem}

\begin{proof}[Proof of Corollary~\ref{ci}.] By hypothesis we have $\Si=\Si_g(2c)$
  for some $g$ and $c$, and the representation under consideration is simply $\BS
  (\Si)\otimes \mathbb{C}$. As 
$1-\zeta_p$  
is a unit in
  $\mathbb{C}$, Theorem~\ref{mainth} shows that the group
  algebra $\mathbb{C}[\widetilde
  \Gamma_\Si^{++}]$ of the extended mapping class group 
maps onto $\End_\mathbb{C}(\BS(\Si)\otimes
\mathbb{C})$. Thus Lemma~\ref{onto} gives the result.
\end{proof}

The same argument works in finite characteristic  $\neq p$, as follows:

\begin{cor}\lbl{irred} If $I$ is an ideal in $\BO=\BZ[\zeta_p]$ so
  that   
      $\BO\slash I$
is a finite field of characteristic
  $\ell\neq p$, then the modular representation $\BS(\Si)\slash I
  \BS(\Si)$ 
in characteristic $\ell$ of the extended mapping class group 
is irreducible. 
\end{cor} 
 
\begin{proof} The hypothesis implies that $1-\zeta_p$  
is a unit in 
      $\BO\slash I$ 
and so the result follows from
Theorem~\ref{mainth} as in
the complex case above.
\end{proof}

\begin{rem}{\em More generally, if $I$ is any ideal in $\BO$ 
not divisible by
$(1-\zeta_p)$, then $1-\zeta_p$  
is a unit in the ring 
      $R=\BO\slash I$
and so the image of the
extended mapping class group generates the entire endomorphism
ring \hbox{$\End_R(\BS(\Si)\slash I
  \BS(\Si))$.} This is why we focus our attention on the ideal  $(1-\zeta_p)$ in
  this paper.
}\end{rem}

\section{Skein-theoretic definition of the integral TQFT representation}
\lbl{sec3}  
The theory of Integral $\SO(3)$-TQFT is slightly different depending on the
congruence class of the prime   $p \pmod 4$. Following \cite{GM}, let  $\BO_p=\BZ[\zeta_p]$, if
$p\equiv -1 \pmod 4$, and $\BO_p=\BZ[\zeta_{4p}]$, if $p\equiv 1 \pmod 4$.
The Integral TQFT lattice $\BS_p(\Si)$ defined in \cite[\S2]{GM} has
coefficients in the ring $\BO_p$. In the case $p\equiv 1 \pmod 4$, we defined in \cite[\S13]{GM} also another 
lattice $\BSplus(\Si)$, with coefficients in $\BZ[\zeta_p]$. The
lattice $\BSplus(\Si)$ 
lattice is again free of finite rank. It is contained (as a set) in
$\BS_p(\Si)$, and one has 
\begin{equation}\lbl{tens}\qquad \qquad \qquad 
  \BSp(\Si)=\BSplus(\Si)\otimes_{\BZ[\zeta_p]}\BZ[\zeta_{4p}]   \qquad \qquad (p\equiv 1 \pmod 4)
\end{equation} where $\zeta_{4p}^4=\zeta_p$. 
The `bigger' lattice $\BS_p(\Si)$ is in some sense the more natural one from the TQFT point
of view. But for the purpose of studying mapping class group representations, we want the coefficient ring to be as small as
possible. Therefore we take the lattice $\BS(\Si)$ considered in
Theorem~\ref{mainth} to be  
\begin{equation}\lbl{tens2}
\BS(\Si)= \begin{cases} \BSp(\Si) & \text{if $p\equiv -1 \pmod 4$}\\
\BSplus(\Si)& \text{if $p\equiv 1 \pmod 4$}
\end{cases}
\end{equation}

Thus, in both cases, $\BS(\Si)$ is a free $\BO$-module of finite rank,
where 
$\BO=\BZ[\zeta_p]$. We will continue to use the notations $\BO$ and
$\BS(\Si)$ whenever it is possible to make statements which are true
in both cases. 
However, the proof of Theorem~\ref{mainth} will require some additional
arguments in the case $p\equiv 1 \pmod 4$. We hope there will be no
confusion from the fact that $\BO$ is different from $\BO_p$ in this case.

The construction of these lattices  in \cite{GM} is a refinement of the
skein-theoretic 
construction of TQFT from the Kauffman bracket
in \cite{BHMV2}.

\begin{notation}
 Throughout the rest of the paper, we use the notations 
$h=1-\zeta_p$ and  
   $d=(p-1)/2$. 
\end{notation}

We let Kauffman's skein variable be $A= - \zeta_p^{d+1}$; this is a primitive $2p$-th
 root of unity, and a square root of $\zeta_p$. 
 As customary in the 
skein-theoretic
 approach to TQFT, 
rather
 than considering surfaces with boundary, we think of closed surfaces
 equipped with colored banded points. (A banded point in a surface $\Si$ is a
 small oriented closed interval embedded in $\Si$.) Thus, 
$\Si=\Si_g(2c)$  
will from now on stand for a closed surface of genus $g$ equipped with
one colored 
banded point colored $2c$, where $0\leq c\leq d-1$.

The TQFT-module $V_p(\Si)$ of \cite{BHMV2} has
coefficients in $\BO_p[h^{-1}]= \BO_p[p^{-1}]$.\footnote{We have
  $\BO_p[h^{-1}]= \BO_p[p^{-1}]$ because $(h^{p-1})=(p)$ as ideals in
  $\BZ[\zeta_p]$.} 
  This is a version  of the
Reshetikhin-Turaev $\SO(3)$-TQFT module at the prime $p$ associated to
$\Si$. Elements of $V_p(\Si)$ are represented skein-theoretically as linear combinations of
colored banded graphs  in a handlebody $H$ with $\partial H=\Si$,
where the graphs should meet the boundary nicely
in the colored banded point.  For example, the lollipop tree   
$T_g^{(2c)}$ of Figure~\ref{lol} (with banding parallel to the plane)
and some coloring defines an element of $V_p(\Si)$ where
$\Si=\Si_g(2c)$.

More generally, any closed connected 3-manifold $M$
equipped with an 
identification  of its boundary $\partial M$ with $\Si$, and
containing  
a colored banded graph which meets $\Si$ in the banded point, defines
an element $Z(M,G)$ of $V_p(\Si)$. Here (as is customary) we suppress
the identification  of $\partial M$ with $\Si$ from the
notation, but it is important to realize that the element $Z(M,G)$ depends
on this identification; indeed, 
the extended mapping class group $\widetilde  \Gamma_\Si^{++}$ 
acts on the set of such $(M,G)$ 
by changing how $\partial M$ is identified with $\Si$, 
and this induces the action of $\widetilde
\Gamma_\Si^{++}$ on the TQFT-module $V_p(\Si)$.  
 We also downplay a technicality: in order to define $Z(M,G)$ we must fix a
lagrangian in $H_1(\Si;\BQ)$ 
  and equip $M$ with an integer weight to resolve the so-called {\em framing
  anomaly}.  Increasing the weight by $1$ multiplies  $Z(M,G)$ by an
invertible scalar called $\kappa$ in \cite{GM}. But 
lagrangians and weights will play almost no role in the present paper,
and by abuse of notation, we simply write $Z(M,G)$ and
$V_p(\Si)$.\footnote{Formally, one may think that the notation $M$ 
stands for a
  manifold equipped with an integer weight, and $\Si$ stands for a surface
  equipped with a lagrangian. This is the point of view adopted in
  \cite{GM3}.}
This method of resolving the framing anomaly by using lagrangians and
integer weights was 
pioneered   
by  Walker \cite{W} and developed by Turaev \cite{T}. See also \cite{GM3} for
    our take on this.\footnote{Another way to resolve the framing anomaly  is to use
    $p_1$-structures as in \cite{BHMV2}.}

The Integral TQFT lattice 
$\BS_p(\Si)$ was originally defined as 
 the $\BO_p$-submodule of $V_p(\Si)$ spanned by all the elements $Z(M,G)$
 coming from connected 3-manifolds $M$ equipped with colored graphs
 $G$ (and any integer weight) \cite{G,GM}. 
 In \cite{GMW}, 
it was shown that $\BS_p(\Si)$ could be also described as  
the $\BO_p$-submodule
spanned by certain skein elements in the handlebody  $H.$ In other words, we
can use the 
fixed 3-manifold $H$ if we permit our skein elements  
to have certain carefully controlled denominators. This is done as
follows.

Recall that banded
knots can be 
also 
colored (= cabled) by skein elements (= linear combinations of banded
links) 
living  
in a solid torus. 
The 
skein-theoretic 
definition of $\BS_p(\Si)$ is as the  $\BO_p$-submodule of
$V_p(\Si)$ spanned by 
so-called {\em mixed graphs}\footnote{Mixed
  graphs were called { \em $v$-graphs} in \cite{GM}.} in \cite{GMW}
where a mixed graph is a colored banded graph $G$ union a banded link $L$   with each
component of $L$ colored by the following skein element $v$ in the
solid torus: 
\begin{equation}\lbl{def-v} v=h^{-1} (z+2)
\end{equation}
(Here, $z$ stands for the banded knot given by the core of the solid
torus.)
It is important to observe that the $\BO_p$-module $\BS_p(\Si)$
obtained in this way is bigger than the $\BO_p$-submodule of $V_p(\Si)$ spanned by colored
banded graphs in $H$ because $v$ has denominator $h$, which is a prime in
$\BO$. 
But it is this  $\BO_p$-module  $\BS_p(\Si)$ which carries a natural action of the extended mapping class group  $\widetilde  \Gamma_\Si^{++}$, 
as is most easily seen from the above description of $\BS_p(\Si)$  in
terms of all connected 3-manifolds $M$ with boundary identified with
$\Sigma.$ 

According to 
(\ref{tens}) and (\ref{tens2}),
in the case $p\equiv 1 \pmod 4$,
$\BS(\Si)$ is 
an $\BO$-module contained in the $\BO_p$-module $\BS_p(\Si).$
We still have that $\BS(\Si)$ is  the $\BO$-submodule of
$V_p(\Si)$ spanned by  
all mixed graphs 
in $H$ \cite[\S13]{GM}.

We refer the reader to \cite{GM} for the   
definition of the
basis elements 
$\bb_\sigma$ of 
$\BS(\Si)$ discussed in the introduction.  
Details of this definition will not be needed in this paper, but we
remark that 
although  
the $\bb_\sigma$
are {\em indexed} by colorings of the banded lollipop tree   
$T_g^{(2c)}$, the basis element  $\bb_\sigma$ is not just $T_g^{(2c)}$
with coloring $\sigma$, but 
is 
 a mixed graph in the sense explained
above,
which furthermore 
has been rescaled by a certain non-positive power of $h$.
In fact, the colorings of $T_g^{(2c)}$ give the so-called {\em
  graph basis} of $V_p(\Si)$, and the basis $\{\bb_\sigma\}$ of
$\BS(\Si)$ is obtained from this graph basis by a triangular (but not
unimodular) base
change (see \cite{GM} for more details).

Let us now describe the action of the extended mapping class group on
$\BS(\Si)$. First we recall that, as usual in TQFT, any cobordism $M$
from $\Si$ 
to itself, equipped with some colored graph $G$ connecting the colored
point in the source surface $\Si$ to the colored point in the target
surface $\Si$, defines an endomorphism 
 of $V_p(\Si)$ which, by abuse of notation, we denote again by
$Z(M,G)$.     
If furthermore $M$ is connected, then $Z(M,G)$ preserves the lattice
$\BS_p(\Si)$ (see \cite{GM}).\footnote{But $Z(M,G)$ does not
  necessarily preserve the lattice $\BS(\Si)=\BSplus(\Si)$ in the case $p\equiv 1 \pmod
  4$. We will address this problem at the end of \S\ref{sec5}.}  
For example, consider $M=\Si\times\I$ 
with weight zero, 
equipped with the  banded arc $pt
  \times \I$ with color $2c$ (where $pt$ is the colored
banded point in $\Si$).
We denote this `vertical' 
banded arc by
 $\varc$, 
and we let $\varc(2c)$ denote this arc colored by $2c$. Then we have
that $Z(\Si\times\I, \varc(2c))$  is the identity map of $\BS(\Si)$.

Next recall that the mapping class group $\Gamma_\Si$ is generated by
Dehn twists about simple closed curves $\alpha$ which avoid the
colored banded point. (We remind the reader that $\Gamma_\Si$ here is the
mapping class group of $\Si$ which fixes the banded point, which is
the same as the mapping class group rel.~boundary of $\Si$ with an open disk around the banded point
removed; in particular, the Dehn twist about a curve
encircling the banded point is non-trivial in $\Gamma_\Si$.) For a
certain skein element $\omega_+$ in the solid torus to be specified
below, we denote
by $\alpha_0(\omega_+)$ the skein element in $\Si\times\I$ obtained 
by cabling  
$\alpha \times \frac 1 2 $ by $\omega_+$, with framing zero relative to
the surface $\Si\times \frac 1 2$ lying in $\Si\times\I$.
Here is, then, a 
skein-theoretic 
description of the action.

\begin{prop}\lbl{act} For a certain lift $W(\alpha)$ of the Dehn twist about
  $\alpha$ to the extended mapping class group $\widetilde
  \Gamma_\Si^{++}$, the action of $W(\alpha)$ on $\BS(\Si)$ is
  given by the endomorphism 
\begin{equation}\lbl{act-alpha}
Z(\Si\times\I, \varc(2c) \cup \alpha_0(\omega_+))~.
\end{equation}  
\end{prop} 

Statements like this are, of course, well-known in the  
skein-theoretic 
approach to
TQFT. A more precise formula, taking into
account the extra structure
  (weights and lagrangians) needed to define the central extension $\widetilde
  \Gamma_\Si^{++}$ and to specify the lifts $W(\alpha)$ of Dehn twists to $\widetilde
  \Gamma_\Si^{++}$, can be found in 
\cite[\S11]{GM3}.
We don't need
this more precise formula 
to describe the image of the group algebra $\BO[\widetilde
  \Gamma_\Si^{++}]$ in $\End_\BO(\BS(\Si))$,
because 
(1)
the lifts $W(\alpha)$ generate the extended mapping class group $\widetilde
  \Gamma_\Si^{++}$, and 
(2) 
in the central extension $$\BZ\rightarrow \widetilde
  \Gamma_\Si^{++} \rightarrow \Gamma_\Si~, $$ any two lifts of the same mapping
class to $\widetilde
  \Gamma_\Si^{++}$ differ by a power of the
central generator, and this central generator acts on $\BS(\Si)$ by 
 $\kappa^4$ (which is a unit in $\BO$)  
times the identity map.\footnote{In fact, $\kappa$ is a square root of
  $A^{-6-p(p+1)/2}$,  
and since $A^2=\zeta_p$ we have  
$\kappa^4=\zeta_p^{-6}.$  Notice that $\kappa^4\equiv 1 \pmod
  h$. This is why the representation of $\widetilde
  \Gamma_\Si^{++}$ on $F(\Si)=\BS(\Si)\slash h \BS(\Si)$ factors
  through the ordinary mapping class group  $\Gamma_\Si$.\lbl{foot9}} 
 Thus we have the following corollary, which
will be the starting point for our proof of Theorem~\ref{mainth}.

\begin{cor} \lbl{eom} The image of $\BO[\widetilde
  \Gamma_\Si^{++}]$ in $\End_\BO(\BS(\Si))$ is the $\BO$-subalgebra
  generated by the endomorphisms $Z(\Si\times\I, \varc(2c) \cup
  \alpha_0(\omega_+))$  associated to simple
  closed curves $\alpha$ on $\Si$ avoiding the colored banded point.
\end{cor}

It remains to describe $\omega_+$. Its key property (which basically
implies 
Proposition~\ref{act}) is described in Figure~\ref{figomega}.

\begin{figure}[h]
\includegraphics[width=1.8in]{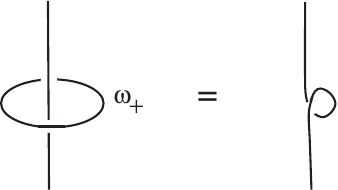}
\caption{Encircling a strand with $\omega_+$ has the same
effect in TQFT as giving that strand a positive twist.} \lbl{figomega}
\end{figure}

An explicit formula for $\omega_+$ is given in
\cite[\S4]{GM2}, based on computations of a similar skein
element in \cite{BHMV1}. 
In this paper, we only need to know the
following about $\omega_+$. Let $\mathcal T$ denote a closed surface
of genus one ({\em i.e.}, a torus) viewed as the boundary of a solid
torus.  Since $\omega_+$ lies in the solid torus, it defines an element of
$V_p({\mathcal T})$. 
By abuse of notation, we denote this element again by
$\omega_+$.

\begin{prop} \lbl{pw} One has $\omega_+\in \BS(\mathcal
  T)$. Moreover,  for the
  usual algebra  structure in $\BS(\mathcal T)$ coming from thinking
  of the solid torus as an annulus $\times\ \I$, the elements
$$1, \omega_+, \omega_+^2, \ldots, \omega_+^{d-1}$$ form a basis of
 $\BS(\mathcal T)$.
\end{prop}
(Here, $1$ stands for the empty link, which is the unit element
of this algebra.)

 \begin{proof} 
We first consider the case $p\equiv -1 \pmod 4$, where $\BS(\mathcal
T)= \BS_p(\mathcal T)$. 
The following proof is similar to the proof of
Theorem 6.1 of \cite{GMW}, where the reader  will find more details on
some of the arguments that follow.

By its very definition \cite{G,GM}, the lattice $\BS_p(\mathcal T)$
contains the TQFT-vectors associated to all connected
manifolds with boundary $\mathcal T$. 
Since $\omega_+$, up to multiplication by a unit, 
 is also represented in $V_p(\mathcal T)$ by the result of 
$-1$-framed
 surgery to the solid torus with boundary 
$\mathcal T$ along 
the core of the solid torus,
we have that $\omega_+ \in \BS_p(\mathcal T)$.  
Let $P$ be the 
$\BO_p$-module spanned 
by the purported basis.  By similar reasoning, $P \subset \BS_p (\mathcal T)$.

Let
$\{ e_0=1,e_1,  
\ldots, 
e_{d-1} \}$ 
be the standard graph basis 
of
$V_p(\mathcal{T}).$ The determinant of  a matrix which  
expresses
a basis
for $\BS_p(\mathcal T)$ in terms of  
$\{ e_0,e_1, \ldots, e_{d-1} \}$ is
given by $h^{-d(d-1)/2}$, up to units  (see \cite[p.~272]{GMW}).

Consider the Hopf pairing $(( \ , \ )): \BS_p(\mathcal T) \times
\BS_p(\mathcal T) \rightarrow \BO_p$, where $(( x , y ))$ is given by
cabling the zero-framed Hopf link with $x$ and $y$ and then evaluating
the resulting skein 
element  
in $S^3$ 
(normalized so that the empty link evaluates to $1$.)  
 The determinant of the $d \times d$ matrix $((e_i , e_j ))$
is up to units given by $h^{d(d-1)}$.
This is because the
$e_i$ are an orthonormal basis for $V_p(\mathcal{T})$ with respect to
the usual Hermitian TQFT form on $V_p(\mathcal{T})$ \cite{BHMV2}, and
the Hopf pairing that we are considering differs from this by  
complex
 conjugation (which leaves the $e_i$ fixed), the action of   
a homeomorphism (which is an isometry), and a rescaling by  $h^{d-1}$
(up to units). 

If we pair our purported basis with the $e_i$ under the Hopf pairing, we
get $$((\omega_+^i,e_j)) = (-1)^j [j+1]\mu_j^i~,$$  
where 
$\mu_j=   
(-A)^{j(j+2)} 
=  \zeta_p^{(d+1)j(j+2)}$
 are the twist eigenvalues,
and the quantum integers $[j+1]$ are defined by 
$$[n]
 =\frac {A^{2n}-A^{-2n}}{A^{2}-A^{-2}}=\frac {\zeta_p^n-\zeta_p^{-n}}{ \zeta_p-\zeta_p^{-1}} ~. $$
Ignoring the unit\footnote{The fact that the quantum integers $[j+1]$
  appearing here are units in $\BZ[\zeta_p]$ is shown in \cite[Lemma~3.1(ii)]{MR2}.} 
factors  $(-1)^j [j+1]$  which appear as multiples of the columns, 
this is a Vandermonde matrix.
(This matrix appeared on \cite[p.~272]{GMW}.)  
Up to units,  its determinant is $h^{d(d-1)/2}$. It follows that the
determinant of the matrix expressing $\omega_+^i$ in terms of $e_j$
is given by $h^{-d(d-1)/2}$, up to units. As a known basis 
for $\BS_p(\mathcal T)$ has this same property, and  $P \subset \BS_p
(\mathcal T)$, it follows that  $P = \BS_p (\mathcal T)$.  
This completes the proof in the case $p\equiv -1 \pmod 4$. 

Now assume $p\equiv 1 \pmod 4$. Then the exact same proof as above shows that $\{1,
\omega_+, \omega_+^2, \ldots, \omega_+^{d-1}\}$ is a basis of
 $\BS_p(\mathcal T)$, which now has coefficients in
 $\BZ[\zeta_{4p}]$. But the explicit formula for $\omega_+$ given in
\cite[\S4]{GM2} shows that $\omega_+$ and its powers lie in the $\BZ[\zeta_{p}]$-lattice
 $\BS(\mathcal T)=\BSplus(\mathcal T) $.  Using (\ref{tens}), it follows that  $\{1,
\omega_+, \omega_+^2, \ldots, \omega_+^{d-1}\}$ is a basis of
 $\BS(\mathcal T)$.  This completes the proof in the case $p\equiv 1 \pmod 4$.  
\end{proof}

\section{$v'$-colored links in the product of a surface and  an interval}

 As the material
  presented
   here has independent interest, for this section only,
 we work in a  more general context where
the coefficient ring can be any integral domain $R$
containing an invertible element $A$
for which $1+A \ne 0$.

Let $S$ be a compact oriented surface, possibly  with boundary (but
without colored points). Let ${\mathcal K}(S\times \I)$ be the Kauffman bracket
skein module of $ S\times \I$ with coefficients in $R$. 
Observe that ${\mathcal K}(S\times \I)$ has a natural
product structure 
 (given by stacking one banded link on top of another)
which makes ${\mathcal K}(S\times \I)$ into an algebra.  
A banded link $L$ is called {\em layered} if it can be written as a
product of banded knots, {\em i.e.,} if each component of $L$  
projects to  a different subset of $\I$.
We say that a banded
knot 
 in $S\times \I$ is {\em flat} if it is (up to isotopy) entirely contained in a
surface $S\times \{t\}$ for some $t$  in the interior of
$\I$,  and the banding is also flat (i.e. parallel to the layers).   
Equivalently, a banded 
knot
 is flat if it has a diagram on
$S$
without crossings.

Let $v'$ be the following slight modification\footnote{Note that $v'$ was called $v$ in
\cite{GMW}, but in the present paper we follow \cite{GM} and reserve the notation $v$
for the element defined in Equation (\ref{def-v}). We will see in
(\ref{vvp}) that $v$ and
$v'$ are essentially equivalent when $R=\BO$.}  
of the element $v$
defined in (\ref{def-v}): 
\begin{equation*} 
\lbl{def-v'} v'=\frac {z+2}{1+A}
\end{equation*}
Note that $v'$ lies in  ${\mathcal K}(\mathrm {solid \
  torus})\otimes R[(1+A)^{-1}]$.

\begin{lem} The $R$-submodule of ${\mathcal K}(S\times \I)\otimes
  R[(1+A)^{-1}]$ spanned by $v'$-colored banded links is generated, as
  an algebra, by $v'$-colored  banded knots.  
\end{lem}

\begin{proof}\lbl{L1}
We need to convert a 
 $v'$-colored 
banded link 
$L$  
to an $R$-linear combination of banded links 
where the components are layered in some order one on top of the
other.  
Recall 
that one of the Kauffman bracket skein relations  \cite{Ka} is

\[ 
\begin{minipage}{0.2in}\includegraphics[width=0.2in]{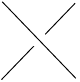}\end{minipage}
=  
 \ \ A \  \ 
\begin{minipage}{0.2in}\includegraphics[width=0.2in]{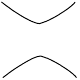} \end{minipage}
+
\ \  A^{-1} \  \ \begin{minipage}{0.2in} \includegraphics[width=0.2in]{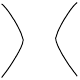}  \end{minipage}
   \]

The following equation shows how to 
use  
this
skein relation to  
change
 crossings of strands belonging to different 
connected components of $L$  
at the cost of
 introducing an $R$-linear combination of $v'$-colored banded links with fewer crossings and fewer components.  The dotted lines show the connection scheme in the complement of a disc in $S$ where the crossing occurs. This allows a proof by induction
 on the number of 
connected components of $L$.
\begin{align*}  
  \begin{minipage}{0.6in}\includegraphics[width=0.6in]{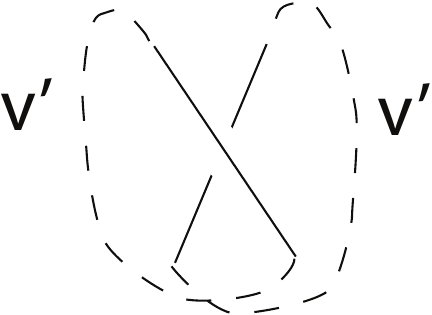}\end{minipage} -
 \begin{minipage}{0.6in}\includegraphics[width=0.6in]{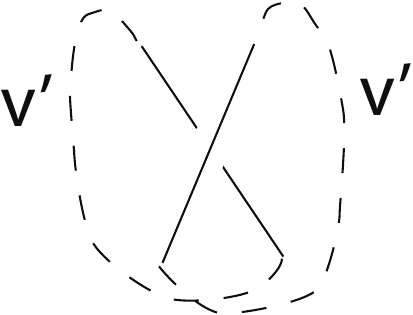}\end{minipage}
&= \frac {1}{(1+A)^2} \left(
  \begin{minipage}{0.6in}\includegraphics[width=0.6in]{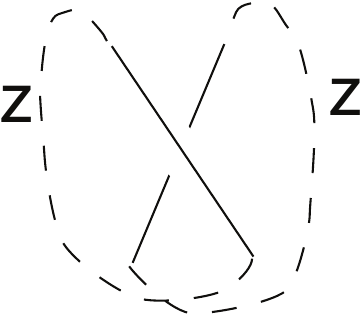}\end{minipage} -
  \begin{minipage}{0.6in}\includegraphics[width=0.6in]{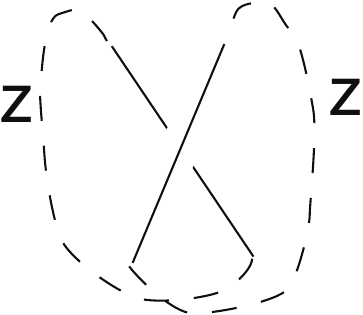}\end{minipage}
  \right)
  \\
  &
 = \frac {A - A^{-1}}{(1+A)^2} \left(
  \begin{minipage}{0.6in}\includegraphics[width=0.6in]{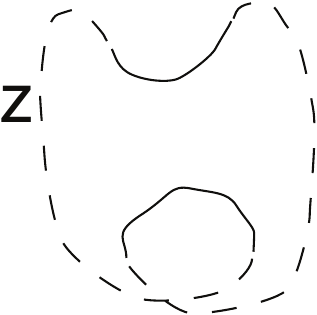}\end{minipage}-
 \begin{minipage}{0.6in}\includegraphics[width=0.6in]{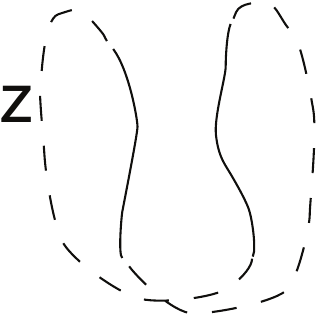}\end{minipage} \right)\\
 &= \frac {A - A^{-1}}{1+A} \left(
\begin{minipage}{0.6in}\includegraphics[width=0.6in]{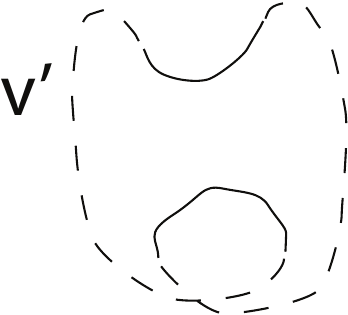}\end{minipage}-
 \begin{minipage}{0.6in}\includegraphics[width=0.6in]{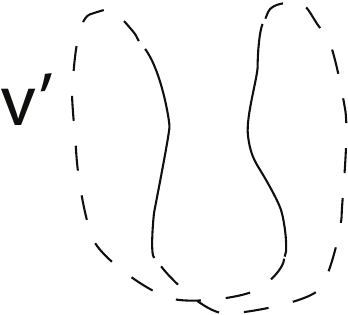}\end{minipage} \right)
 \\
 &
 = (1-A^{-1}) \left(
\begin{minipage}{0.6in}\includegraphics[width=0.6in]{av.pdf}\end{minipage}-
 \begin{minipage}{0.6in}\includegraphics[width=0.6in]{bv.pdf}\end{minipage} \right)
  \end{align*}

  \end{proof}

\begin{prop} \lbl{flatv'}The $R$-submodule of ${\mathcal K}(S\times \I)\otimes
  R[(1+A)^{-1}]$ spanned by $v'$-colored banded links is generated, as
  an algebra, by $v'$-colored flat banded knots.  
\end{prop}

\begin{proof}
 Using the previous 
lemma,
it is enough to see that an element
 given by a single  $v'$-colored 
banded
knot can be written as an $R$-linear combination of products of flat banded knots.  
This can be proved 
by induction on the number of crossings of the banded knot, using the following 
 equation, and a similar one obtained 
from the 
same
starting diagram but with the opposite crossing data. The two asterisks are meant to indicate two points on a disk in  $S$ to help locate the placement of the links with respect to these  reference points.
Thus  the empty link could be denoted by two asterisks, but to save space, we write simply a scalar for  
that scalar times the empty link. The dotted lines play the same role
as in the previous equation. 

The following equation, then,
shows that a $v'$-colored banded knot with $n$ crossings can be
rewritten as an $R$-linear combination of three $v'$-colored banded knots
with fewer than $n$ crossings, one $v'$-colored two-component
banded link
also with fewer than $n$ crossings,
and the 
empty link.   
 By the same reasoning as in the proof of the previous lemma, 
this $v'$-colored two-component 
banded link may be written
as a linear combination of $v'$-colored layered links with fewer than  
$n$ crossings. 
So the needed result can be proved  
by induction on the number of crossings.
  
   \begin{align*}
   \begin{minipage}{0.4in}\includegraphics[width=0.4in]{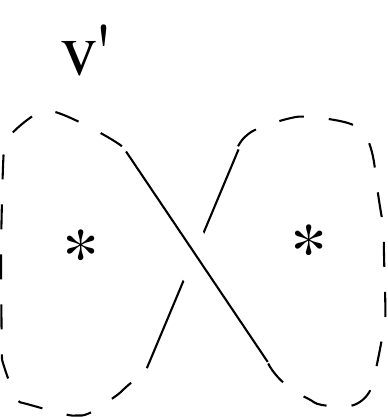}\end{minipage}
&= \frac {1}{1+A} \left(
  \begin{minipage}{0.4in}\includegraphics[width=0.4in]{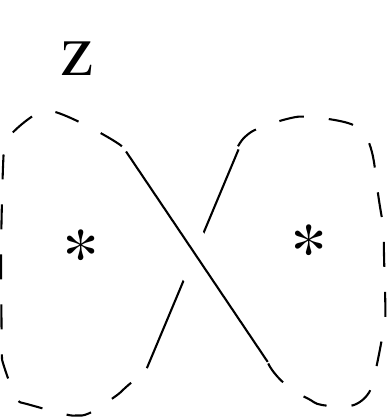}\end{minipage}
\ + \ 2\right)\\
&=
   \frac {1}{1+A} \left( A\ 
  \begin{minipage}{0.4in}\includegraphics[width=0.4in]{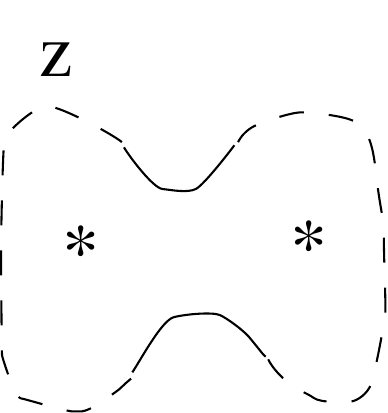}\end{minipage}
\ + \ 
  A^{-1} \ 
  \begin{minipage}{0.4in}\includegraphics[width=0.4in]{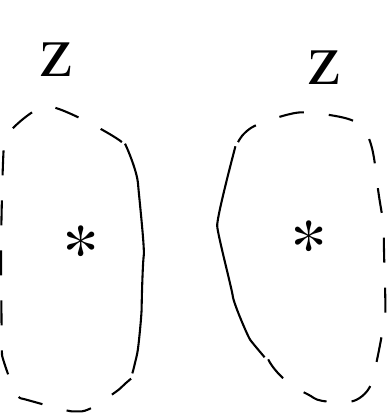}\end{minipage}
\ + \ 2\right)\\
&=
 \frac {1}{1+A} 
 \Bigg( A\ 
  \begin{minipage}{0.4in}\includegraphics[width=0.4in]{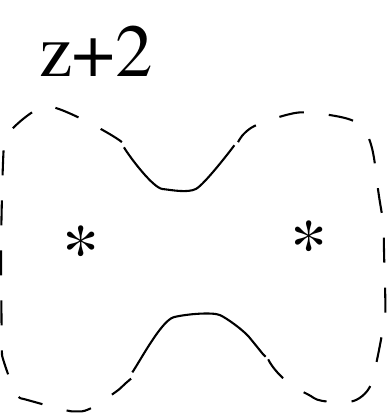}\end{minipage}
\ + \ 
A^{-1} \ 
\begin{minipage}{0.4in}\includegraphics[width=0.4in]{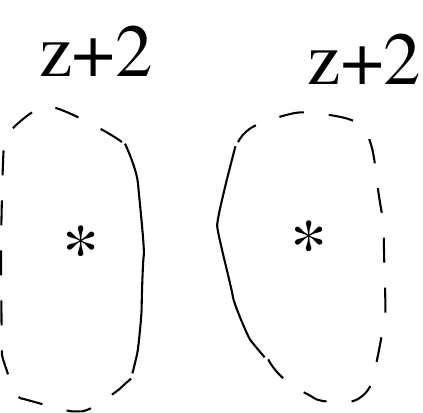}\end{minipage}
\  
 -\ 2 A^{-1} \ 
\begin{minipage}{0.4in}\includegraphics[width=0.4in]{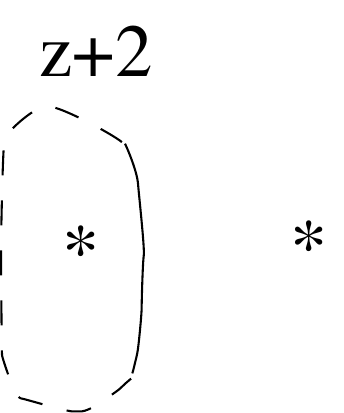}\end{minipage}
 \\
& \hspace{5cm} 
-\ 2 A^{-1}\  
\begin{minipage}{0.4in}\includegraphics[width=0.4in]{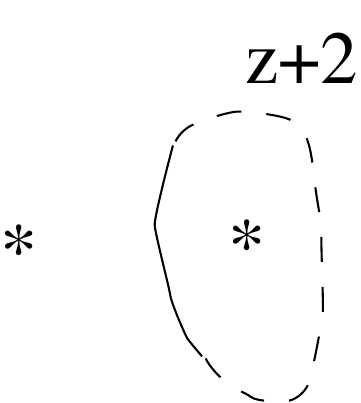}\end{minipage}
\  -2A +4 A^{-1}+2
 \Bigg)
\\
 &= 
  A \ \begin{minipage}{0.4in}\includegraphics[width=0.4in]{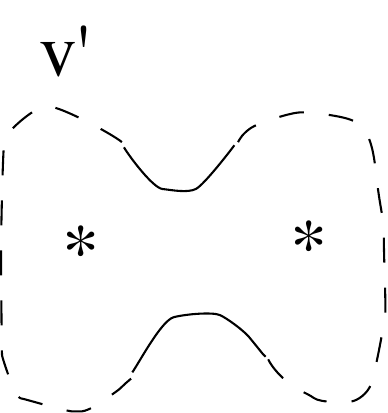}\end{minipage} 
 \ + \ (1+A^{-1})\ 
 \begin{minipage}{0.4in}\includegraphics[width=0.4in]{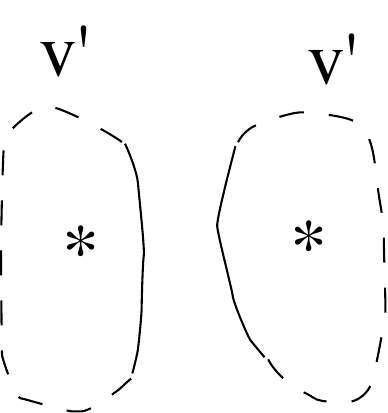}\end{minipage} 
 \ -\ 2 A^{-1} \ 
\begin{minipage}{0.4in}\includegraphics[width=0.4in]{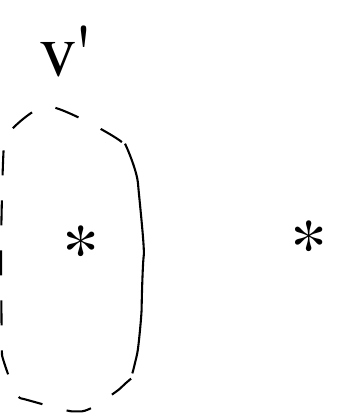}\end{minipage} \\
& \hspace{5cm}
 -2 A^{-1}\  
\begin{minipage}{0.4in}\includegraphics[width=0.4in]{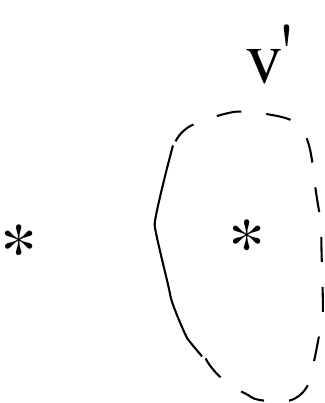}\end{minipage} 
 \ -2 +4 A^{-1}
  \end{align*}
  \end{proof}
  
We will say a banded link in $S\times \I$ is a 
 {\em flat layered} banded link, if 
it is a product of flat banded knots. In other words, 
 each component lies in a different 
$S \times \{t\}$ , the banding is also flat ({\em i.e.,} parallel to the layers), and each component is flat.
 We can rephrase the above result:

\begin{cor} \lbl{clflatv'} The $R$-submodule of ${\mathcal K}(S\times \I)\otimes
  R[(1+A)^{-1}]$ spanned by  $v'$-colored banded links is spanned  
(as an $R$-module) by flat layered $v'$-colored banded links. \end{cor}

\section{Proof of Theorem~\ref{mainth}}\ 
\lbl{sec5}

Let us consider the results of the previous section but taking $R$ to
be $\BO$ with $A= - \zeta_p^{d+1}$. 
 We write
 $x\sim y$ (where $x$ and $y$ lie in some $\BO$-module) if  
$x=uy$  
where $u$ is a unit of $\BO$. Since $1-A$ is a unit in $\BO$
\cite[Lemma~4.1(i)]{GMW}, we have 
$h= 1-\zeta_p =1-A^2 \sim1+A$. 
 Comparing  (\ref{def-v}) and (\ref{def-v'}), we see that
 \begin{equation}\lbl{vvp}
v \sim v'
\end{equation} in ${\mathcal K}(
\mathrm {solid \ torus}
)\otimes \BO[h^{-1}]$ . Thus we have the following specialization of Corollary \ref{flatv'}, where we have permissibly substituted $v$ for $v'$.

\begin{prop} \lbl{cflatv} The $\BO$-submodule of ${\mathcal K}(S\times \I)\otimes
  \BO[h^{-1}]$ spanned by  $v$-colored banded links is spanned by flat
  layered $v$-colored banded links.  
\end{prop}

We now return to 
$\Si=\Si_g(2c)$.  
In order to prove
Theorem~\ref{mainth}, we must 
compute the image of $\BO[\widetilde \Gamma_\Si^{++}]$ in
$\End_\BO(\BS(\Si))$. We proceed in three steps. The first step is
to apply the previous Proposition~\ref{cflatv} with $S$
equal to $\Si$ with a disk around the colored banded point
removed. Thus $S\times\I$ is $\Si\times \I$ minus a tubular
neighborhood of the vertical arc $\varc= pt \times \I$.

\begin{lem} \lbl{ve} The image of $\BO[\widetilde
  \Gamma_\Si^{++}]$ in $\End_\BO(\BS(\Si))$ is equal to the
  $\BO$-submodule 
of $\End_\BO(\BS(\Si))$ spanned by elements of the form
$Z(\Si\times\I, \varc(2c) \cup s),$ where $s$ is some $v$-colored
banded link
in 
$(\Si\times\I) \setminus  \varc.$
\end{lem} 

\begin{proof} Let $\cE$ be the image of $\BO[\widetilde
  \Gamma_\Si^{++}]$ in $\End_\BO(\BS(\Si))$, and let $\cE'$ be the
  $\BO$-submodule of $\End_\BO(\BS(\Si))$ described in the lemma. Note
  that both $\cE$ and $\cE'$ are sub-algebras of
  $\End_\BO(\BS(\Si))$. We
  must show that $\cE=\cE'$.

The inclusion $\cE \subset \cE'$   
is easy: By Corollary  \ref{eom}, we know that $\cE$ is  the $\BO$-subalgebra
  generated by the endomorphisms $Z(\Si\times\I, \varc(2c) \cup
  \alpha_0(\omega_+))$  associated to simple
  closed curves $\alpha$ on $\Si$ avoiding the colored banded point.
Notice that $\alpha_0$ is a flat banded knot, so that $\cE$ is the
subalgebra generated by the endomorphisms coming from flat
$\omega_+$-colored banded knots.    
  As $\omega_+ \in S({\mathcal T})$, and $S({\mathcal T})$ is spanned
  by $v$-colored  banded links
in the solid torus,  
we  see that $\cE \subset \cE'$.

For the opposite inclusion, we must show that  every  $v$-colored
banded link $s$ in 
 $(\Si\times\I) \setminus  \varc$   
can be rewritten as an $\BO$-linear combination of products of flat
$\omega_+$-colored banded knots. We proceed as follows. First, Proposition
\ref{cflatv} tells 
us that $s$ 
can be rewritten as an $\BO$-linear combination of flat layered
$v$-colored banded links, {\em i.e.,} products of flat $v$-colored
banded knots. Thus, it is enough to show that a flat
$v$-colored banded knot can be rewritten as an $\BO$-linear
combination of products of flat 
$\omega_+$-colored banded knots. But this follows from Proposition
\ref{pw}. This completes the proof.
\end{proof}

In the next step of the proof of Theorem~\ref{mainth}, we wish to
replace $\Si\times \I$ by another cobordism from $\Si$ to itself. For
this we use the following general fact about 
Integral
TQFT.

\begin{lem}\lbl{sv2} Let $M$ be a compact connected oriented 3-manifold
  with $\partial M = -\Si \sqcup   \Si$, and let $G
  \subset M$ be a colored banded graph which meets the boundary at the
  colored points of $-\Si$ and $\Si$. 
   Let $(M', G')$ be obtained from $(M,G)$ by
  surgery on some framed link in $M 
\setminus  G$. 
 Then the set of endomorphisms 
$$\{Z(M, G \cup s) \,\vert\, \text{$s$ is a $v$-colored
banded link in $M \setminus  G$}\}$$ spans the same $\BO_p$-submodule of
$\End_{\BO_p}( \BS_p(\Si) )$ as  the set of endomorphisms  
$$\{Z(M', G' \cup s) \,\vert\, \text{$s$ is a $v$-colored
banded link in $M' \setminus  G'$}\}~.$$
\end{lem}

\begin{proof}
This is similar to  \cite[Proposition~1.9]{BHMV2} in the
skein-theoretic 
construction of TQFT over a field. The proof is based on 
 the surgery axiom  
 (see {\em
    e.g.}  
\cite[Lemma~11.1]{GM3})
which says that the effect of performing surgery  along a framed curve
in $M$   is the same 
  as the effect of cabling that curve  with 
a certain skein element $\omega$. 
Given the fact that $\omega$ lies in $S_p({\mathcal T})$ and can
therefore  be  expressed as a $\BO_p$-linear combination of $v$-colored
links
in a solid torus, 
and the fact that surgery is reversible (we can also perform
surgery $M'$ in the complement of $G'$ to recover $M$), the result
follows.
\end{proof}

Notice that we have formulated the above Lemma~\ref{sv2}  for the lattice
$\BS_p(\Si)$. In the case $p\equiv 1\pmod
4$ (where $\BO\subsetneq \BO_p$ and $\BS(\Si)\subsetneq  \BS_p(\Si)$),
the
  lemma as stated does not hold for $\BS(\Si)$.
  This is not really a problem, as there is also a version of
Lemma~\ref{sv2} for the lattice $\BS(\Si)$.
 But we defer the discussion of how to deal with this issue to the end
 of this section. 
Therefore, in what
follows, we will work with the lattice $\BS_p(\Si)$. We will thus
obtain a proof of Theorem~\ref{mainth} for $\BS_p(\Si)$ in place of
$\BS(\Si)$. This will already constitute a proof of Theorem~\ref{mainth} in the case $p\equiv -1\pmod
4$. Once this is done, we will then explain the extra arguments needed in the case $p\equiv 1\pmod
4$  to conclude the proof.

In the second step of the proof of Theorem~\ref{mainth}, we use  
Lemma~\ref{sv2}  
to replace $\Si\times \I$ with the interior
connected sum of two handlebodies, $H_1\# H_2$, 
where $H_2$ is a
genus $g$ handlebody with boundary $\Si$,  and $H_1$ is $H_2$ with
the reversed orientation. Note that $H_1\# H_2$ is a cobordism
from $\Si$ to itself. 

We will need an explicit description of this construction. To this end, we identify $\Si\times\I$ with the
complement of a neighborhood of two 
linked graphs $G_1, G_2 \subset S^3$, as drawn in Figure
\ref{G1G2}. In this figure, the vertical arc $\varc$ is drawn as a thickened line. 
Note that $C$ is not part of $G_1$ nor of $G_2$, as $C$ lies in
$\Si\times \I$.

\begin{figure}[h]
\includegraphics[width=1.8in]{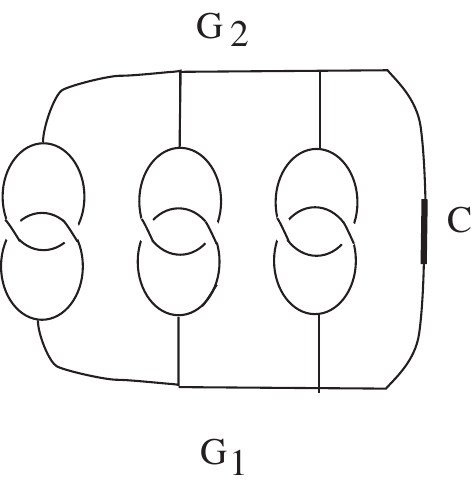} \quad \quad   \includegraphics[width=1.8in]{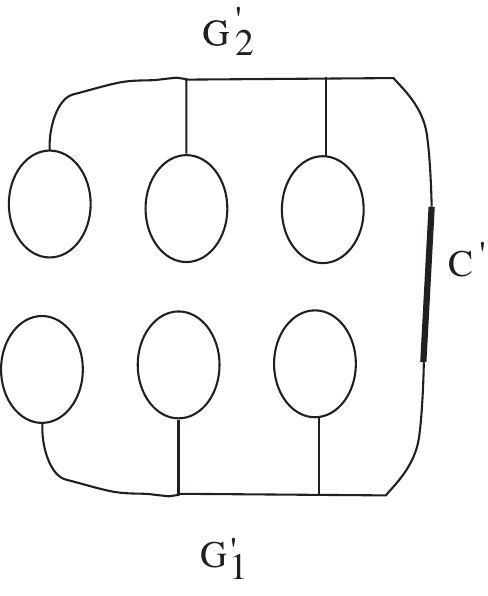}
\caption{$\Si \times I$ can be surgered to obtain $H_1{\#}H_2$ } \lbl{G1G2}
\end{figure}

Also we identify $H_1\# H_2$  with the complement of a
neighborhood of the disjoint union of two graphs $G'_1\sqcup G'_2 \subset
S^3$. One may pass from $\Si\times\I$ to $H_1\# H_2$ by doing
$+1$ framed surgery along $g$ curves which encircle
each of the clasps between $G_1$ and $G_2$. The thickened  line
labelled $\varc'$ in Figure~\ref{G1G2} represents the image of $\varc$ after the surgery.  Again we
remark that $\varc'$ lies in $H_1\# H_2$ and is
not part of $G'_1$ nor of $G'_2$. Note that $\varc'$ meets
the $2$-sphere along which the connected sum occurs  in a single
point.

Lemmas \ref{ve} and \ref{sv2} allow one to conclude that: 
\begin{lem}\lbl{esvs} The image of $\BO_p[\widetilde
  \Gamma_\Si^{++}]$ in $\End(\BS_p(\Si))$ is  the $\BO_p$-submodule
  spanned by 
 the endomorphisms  
$Z(H_1\# H_2, \varc'(2c) \cup s),$ 
where $s$ is  
any 
$v$-colored
banded link in $H_1\# H_2\setminus  \varc'.$
\end{lem}

We can describe this submodule as follows. First, note that $G'_1$ and  $G'_2$ are lollipop trees for two
handlebodies  $\BH_1$ and  $\BH_2$  
(not to be confused with $H_1$ and $H_2$) 
whose  boundaries are copies
of $\Si$ (and each $G'_i$ meeting the boundary of $\BH_i$ in the
colored point). 
Thus the second
figure in Figure \ref{G1G2} describes a decomposition of $S^3$ into
three pieces: the cobordism
$H_1\# H_2$ and the handlebodies $\BH_1$ and  $\BH_2$. Now let
$\nu(\varc')\subset H_1\# H_2$ be a tubular neighborhood of $\varc'$ which meets the
handlebodies $\BH_1$ and $\BH_2$ along 2-disks. The complement of
$\BH_1 \sqcup \BH_2 \cup \nu(\varc')$ in $S^3$ (which is the same as
the complement of $\nu(\varc')$ in $H_1\# H_2$) is a genus $2g$
handlebody $H$. A lollipop tree $T$  for $H$ is drawn in
Figure \ref{G1G2primeT}.

\begin{figure}[h]
\includegraphics[width= 1.8in]{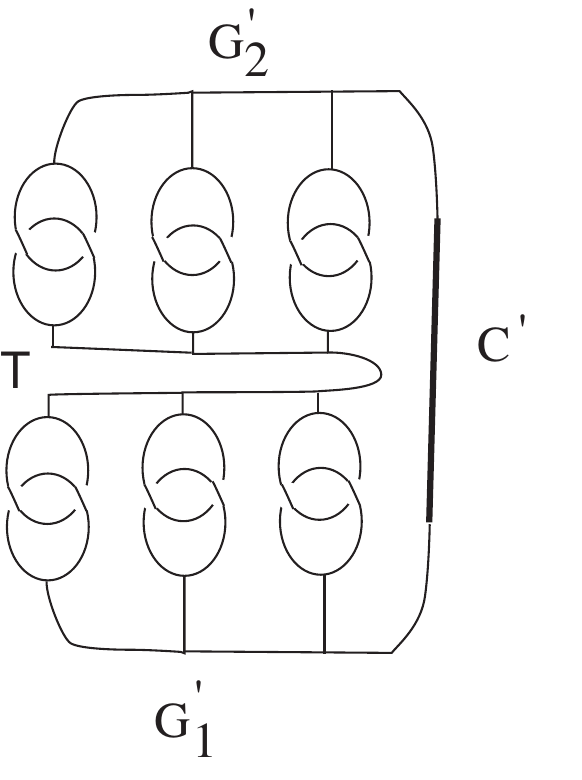} 
\caption{A lollipop tree $T$ for the handlebody  $H$ of genus $2g$
  which is  the complement 
in $S^3$  
of  the union 
of 
neighborhoods of 
 $G'_1$ and $G'_2$ 
and a neighborhood of $\varc'$. } \lbl{G1G2primeT}
\end{figure}

Now observe that  any  $v$-colored banded link in  
      $ H_1 \# H_2 \setminus  \varc'$  
 may be isotoped to  lie in $H$, and the collection of such links then span  
the lattice 
$\BS(\partial H).$ 
The set 
$$\{ \bb_\sigma\, |\,  \sigma \text{ is a coloring of $T$}\}$$ 
(recall that all colorings are assumed small and admissible) is a
basis for  
$\BS(\partial H )$ and hence also for $\BS_p(\partial H ).$
Thus Lemma \ref{esvs} can be restated more
explicitly as follows:

\begin{lem}\lbl{esvl} The image of $\BO_p[\widetilde
  \Gamma_\Si^{++}]$ in $\End(\BS_p(\Si))$ is  the $\BO_p$-submodule
  spanned  by the endomorphisms 
 $$
 Z(H_1 \# H_2, \varc'(2c) \cup \bb_\sigma)~,$$ where $\sigma$ runs
through the colorings of $T$.
\end{lem}

This completes the second step of the proof of
Theorem~\ref{mainth}. For the third and last step, we must first recall the
orthogonal lollipop bases for the integral TQFT modules constructed in
\cite{GM2}. 
The following discussion is valid for the lattices $\BS(\Si)$ as
defined in (\ref{tens2}).  
The elements of an orthogonal lollipop basis are denoted by
$\{\bt_\sigma\}$, and are again indexed by (small admissible) colorings $\sigma$
of a lollipop tree for a handlebody with boundary the given
surface. Of course, for a given surface this basis 
(as well as the original basis $\{\bb_\sigma\}$ 
defined in \cite{GM}) 
depends on a choice
of handlebody and 
lollipop tree with it. The fact that our notation 
in what follows 
does not indicate
this should cause no confusion. Also we note  that, in the proof we
present,  
besides our original surface $\Si$ of
genus $g$ with one  banded point colored $2c$, we also need to
consider $\partial H$, which is a surface of genus $2g$ with no colored points,
and the basis 
of $\BS(\partial H)$  
associated to the lollipop tree $T$ (see Figure \ref{G1G2primeT}).  
For this reason,
in the following discussion of bases, 
we use the letter $2e$ to denote the color of
the colored banded point and thus 
of 
the trunk edge (see 
formulas (\ref{bt}) and (\ref{bts}) below).    In one case, this
trunk color is $2c$ and in the other case 
it 
is zero. The discussion in 
\S\ref{sec.intro}
of bases, when modified in this way, applies as
well to $\BS(\partial H)$.

The important property of the orthogonal lollipop basis for us is that $\{\bt_\sigma\}$ is orthogonal with respect to the Hopf pairing
$$(( \ ,\, )): \BS(\Si)\times \BS(\Si) \rightarrow \BO$$ which is
based on placing skeins in 
neighborhoods of  
linked lollipop trees and evaluating.  See  
\cite[\S3]{GM2}  
for more details about this pairing. 
The linked lollipop trees defining the Hopf
  pairing are as in the left
  part of Figure~\ref{G1G2} in the case where the surface is
  $\Si_g(2c)$, and as in Figure~\ref{G1G2primeT} in the case where the
  surface is $\partial H$.

  In \cite{GM2}, there is 
also 
defined a 
basis
$\{\bt^{\#}_\sigma\}$ for the dual lattice
$$ \BS^\#(\Si)= \{ x \in V(\Si)\, | \, (( x,y)) \in  \BO \ \ \forall y \in
\BS(\Si)\}$$ 
 so that \cite[Remark 3.5]{GM2} $$((\bt_\sigma, \bt^{\#}_{\sigma'}))\sim
\delta^{\sigma}_{\sigma'}~.$$    
 Moreover,  $\bt^{\#}_\sigma$ is a power of $h^{-1}$ times
 $\bt_\sigma$. For our computation below, it
 will be convenient to express this rescaling 
by a power of $h^{-1}$   
as follows. 
Given a coloring $\sigma$ of a lollipop tree, 
let $A(\sigma)=\sum a_i$ denote the sum of
the half-colors 
 at the stick edges 
(see
 Figure~\ref{lol}). Then  
for a certain skein element $x_\sigma$
\cite[Equation (6), Corollary 3.4]{GM2}, we have  
\begin{align} \bt_\sigma & = h^{- \lfloor (A(\sigma)-e)/2
    \rfloor} x_\sigma   \lbl{bt}\\ 
{\bt^\#}_\sigma & = h^{-\lceil (A(\sigma)+e)/2 \rceil} x_\sigma \lbl{bts}
\end{align} 
where $\lfloor x \rfloor$ is the  
greatest
integer $\leq x$, and  $\lceil x \rceil$ is the smallest integer $\geq x.$ 

We are now ready for the final step in the proof of
Theorem~\ref{mainth}. Recall the lollipop tree $T$ in the genus $2g$
handlebody $H$.  
If  $\sigma$ is a coloring of $T$, 
let 
$$ Z(\sigma)=Z(H_1 \# H_2, \varc'(2c) \cup \bt_{\sigma})~.$$
The endomorphisms 
$Z(\sigma)$ span the 
 image of $\BO_p[\widetilde
  \Gamma_\Si^{++}]$ in $\End(\BS_p(\Si))$
  by
  Lemma~\ref{esvl}.
 We shall compute the matrix
  of $Z(\sigma)$ 
with respect to the orthogonal 
  lollipop basis of $\BS_p(\Si)$. 
We shall find that the matrix of $Z(\sigma)$ is either zero or a scalar
multiple of an elementary matrix. 

To carry out this computation, we need to 
fix some
notation. Recall that the orthogonal lollipop basis of $\BS_p(\Si)$
consists of the 
$\bt_{\sigma}$ where $\sigma$ runs through the colorings of $G_1$
with trunk color $2c$. For two such colorings 
$\sigma_1$ and $\sigma_2$,
let
$E_{\sigma_2,\sigma_1}$ be the elementary matrix with
$(\sigma_2,\sigma_1)$ entry equal to one, and all other entries equal
to 
zero. (Thus, the corresponding endomorphism sends $\bt_{\sigma_1}$ to
$\bt_{\sigma_2}$, and all other basis vectors to zero.) Next, as  
$G'_i$ is simply $G_i$ in a different position,  we will identify
colorings of $G_i$ with those of $G'_i$. There is also a graph isomorphism 
$r: G_2 \rightarrow G_1$ that is  given 
by reflection across a horizontal axis. 
If $\sigma$ is a coloring of $G_1$, then  $r$ 
induces a coloring $ \sigma \circ r$ of
$G_2$, which we denote 
by 
$r(\sigma)$. 
Let $T'$ denote the graph
$G'_1\cup\varc' \cup G'_2$. Given two colorings  $\sigma_2, \sigma_1$
of $G_1$ which are colored $2c$ on the trunk edge, let  
${\sigma_1} \#_c {\sigma_2}$ denote the coloring of $T'$ obtained by
coloring $G'_1$ by ${\sigma_1}$, $G'_2$ by $r(\sigma_2)$, and 
$\varc'$ by $2c$. Finally, we call 
 $o: T \rightarrow T'$ 
   the graph isomorphism
that sends a loop to the loop that clasps it.

\begin{lem}\lbl{5.6} Let $\sigma$ be a coloring of $T$. Then
  $ Z(\sigma)=0$ unless $\sigma =o (\sigma_1 \#_c
  \sigma_2)$ for some colorings  
$\sigma_1$  and $\sigma_2$
of
  $G_1$ with trunk edge colored by $2c$. Conversely, if $\sigma =o (\sigma_1 \#_c
  \sigma_2)$, then the matrix of $ Z(\sigma)$ is a scalar multiple of
  the elementary matrix  $E_{\sigma_2,\sigma_1}$. The scalar multiple
  is $h$ times  a unit in $\BO_p$ if $\sigma_1$ is odd and $\sigma_2$
  is even.
(See Definition~\ref{eoc} for the definition of odd and even
colorings.)  
 In
  all other cases, the scalar multiple is a unit.
\end{lem}
\begin{proof} The matrix entry $ Z(\sigma)_{\sigma_2,\sigma_1}$  
is, up to units, 
 the evaluation of the skein
in $S^3$ given by placing $\bt_{\sigma}$ in $H$, $\bt_{\sigma_1}$ in
$\BH_1$, and  $\bt^\#_{r(\sigma_2)}$ in $\BH_2$, with the last two
connected by $\varc' $ 
colored $2c$, and thus determining a coloring of the graph $T'$.  Note that every coloring $\sigma$ of $T$ is of the form $o (\sigma'_1
\#_{c'} \sigma'_2)$ for some $\sigma'_1$, $c'$, and $ \sigma'_2$. Using orthogonality for  
$\partial H$, we see that $ Z(\sigma)_{\sigma_2,\sigma_1}$ is  zero unless
$\sigma'_1= \sigma_1$, $c'=c$, and $\sigma'_2= \sigma_2$. This shows
that $ Z(\sigma)$ is either zero (if $c'\neq c$) or a scalar multiple of the
elementary matrix $E_{\sigma_2,\sigma_1}$ (if $c'=c$ and $\sigma =o (\sigma_1 \#_c
  \sigma_2)$). Assume we are in the latter case. Then the scalar
  multiple can be computed as follows.  Using  (\ref{bt}) and
  (\ref{bts}), and 
$A(\sigma_1\#_c\sigma_2)=A(\sigma_1)+ A(\sigma_2)$,
we have that
 \begin{align*} Z(\sigma)_{\sigma_2,\sigma_1}&=
  h^{- \lfloor (A(\sigma_1)-c)/2 \rfloor} h^{- \lceil (A( \sigma_2)+c)/2 \rceil}
 (( \bt_{ \sigma_1 \#_c \sigma_2}, x_{ \sigma_1 \#_c \sigma_2}))\\
 & = h^{- \lfloor (A(\sigma_1)-c)/2 \rfloor} h^{- \lceil (A(
   \sigma_2)+c)/2 \rceil}
h^{\lceil A(\sigma_1\#_c \sigma_2)/2 \rceil}
 (( \bt_{ \sigma_1 \#_c \sigma_2}, \bt^\#_{ \sigma_1 \#_c \sigma_2}))\\
& \sim h^{ - \lfloor (A(\sigma_1)-c)/2 \rfloor- \lceil (A(
  \sigma_2)+c)/2 \rceil + \lceil (A(\sigma_1)+A(\sigma_2))/2 \rceil }\\
& = 
\left\{\begin{array}{rl}h &\text{if $\sigma_1$ is odd and $\sigma_2$ is even} \\
1 &\text{otherwise}\\
\end{array}\right.
 \end{align*} 
This completes the proof of Lemma~\ref{5.6}.
\end{proof}

The proof of Theorem~\ref{mainth} 
in the case $p\equiv -1 \pmod 4$ (where $\BS(\Si)=\BS_p(\Si)$)
is now 
completed as follows.
The statement to be proved is that the image of the group algebra $\BO_p[\widetilde
  \Gamma_\Si^{++}]$ in $\End(\BS_p(\Si))$ is  
\begin{equation*}
\left[
\begin{array}{cc}
\End(\BS^{ev}_p(\Si))  &\multicolumn{1}{|c}{h \Hom(\BS^{odd}_p(\Si), \BS^{ev}_p(\Si))} \\ \hline
\Hom(\BS^{ev}_p(\Si), \BS^{odd}_p(\Si)) & \multicolumn{1}{|c}{ \End(\BS^{odd}_p(\Si))}
\end{array}\right] 
\end{equation*}
 This is very similar to what Lemma~\ref{5.6} says, except that
Lemma~\ref{5.6} 
expresses $Z(\sigma)$ in  
the orthogonal lollipop basis $\{\bt_\sigma\}$
of $\BS_p(\Si)$, 
whereas   
the definition of
$\BS^{ev}_p(\Si)$ and $\BS^{odd}_p(\Si)$ was in terms of the original
lollipop basis $\{\bb_\sigma\}$. But 
 the 
submodules  
$\BS^{ev}_p(\Si)$ and $\BS^{odd}_p(\Si)$ are the same
 no matter whether we use the original lollipop basis
or the orthogonal lollipop
basis. This is because  \cite[\S2]{GM2} there is a triangular basis
change between the two bases which respects the block summands given
by specifying the stick colors $2a_i$.
 Thus, Lemma \ref{esvl}  together with
Lemma~\ref{5.6} imply Theorem~\ref{mainth} in the case $p\equiv -1
\pmod 4$. 

We now consider the case $p\equiv 1 \pmod 4$. In this case, we need to
be more precise about the integer weights which we put on cobordisms
to resolve the framing anomaly. As mentioned in
\S\ref{sec3}, changing the
weight of a cobordism $(M,G)$ multiplies the induced endomorphism
$Z(M,G)$ by $\kappa$, where $\kappa$ is a unit in $\BO_p$. Thus, as
long as we were working with coefficients in $\BO_p$, we could (and
did) ignore
weights in the above proof. The problem in the case $p\equiv 1 \pmod
4$ is that only $\kappa^2$ lies in $\BO=\BZ[\zeta_p]$, but not
$\kappa$ itself.\footnote
{As already mentioned,  
$\kappa$ is a square root of
  $A^{-6-p(p+1)/2}$, so that $\kappa\in\BZ[\zeta_{p}]$  if and only if
  $p\equiv -1 \pmod 4$. This is why $\BO_p=\BZ[\zeta_{4p}]$ if $p\equiv 1 \pmod
4$. In both cases,  $\BO_p$ is the minimal
  ring containing $\zeta_p$ and $\kappa$.} Thus $Z(M,G)$ preserves the
lattice $\BS(\Si)$ if and only if the weight of $M$ satisfies a parity
condition. This condition is best formulated by saying that  $M$
together with its weight should lie in the {\em even 
cobordism category} defined in \cite{G}. 
If this is the case, we say that the weight of $M$ is of the correct
parity. 
We will not need the
precise formulation of this condition, which is non-trivial 
to state.  (It is {\em not}  
simply to 
require all weights to be even.) 
It will be enough for us to know that 
for every cobordism $M$ and integer $n$, exactly one of $n$ or $n+1$
is a weight of the correct
  parity for $M$. Also, 
for the cobordisms (\ref{act-alpha}) describing the representation of the
extended mapping class group $\widetilde\Gamma^{++}_\Si$ at the
beginning of our proof, it is
alright 
to take the weight to be zero.

Here is, then, how to modify the above proof of Theorem~\ref{mainth}
in the case $p\equiv 1 \pmod 4$ so that it holds for the lattice
$\BS(\Si)$ and not just for the lattice $\BS_p(\Si)$. At the beginning
of the proof, we equip the cobordisms (\ref{act-alpha}) with weight zero. In
Lemma~\ref{sv2}, we add the hypothesis that both $M$ and $M'$ are equipped
with a weight of the correct parity,  
  then the statement holds true for
$\BS(\Si)$ in place of $\BS_p(\Si)$. In Lemmas~\ref{esvs} and~\ref{esvl}, we
equip the cobordism $H_1\# H_2$ from $\Si$ to itself with a weight of
the correct parity, then both Lemmas hold true for the image of $\BO[\widetilde
  \Gamma_\Si^{++}]$ in $\End_\BO(\BS(\Si))$. We do not change anything
  in the computation in Lemma~\ref{5.6}, but notice that $ Z(\sigma)$
  is now defined more precisely (by the requirement on the weight of the cobordism $H_1\# H_2$
  to be of the correct parity).  Lemma~\ref{5.6} shows that the matrix
  of $ Z(\sigma)$
  is either zero or a scalar multiple of an elementary matrix
  $E_{\sigma_2,\sigma_1}$, and it also computes this scalar multiple up
  to a unit in $\BO_p=\BZ[\zeta_{4p}]$. But since we know that the
  matrix of $ Z(\sigma)$  
   has coefficients in $\BO=\BZ[\zeta_{p}]$, this
  actually computes the scalar multiple up to a unit in $\BO$. The
  remaining arguments in the proof are the same as before. 
This completes the proof of Theorem~\ref{mainth}.

\section{Proof of 
Theorem~\ref{new-Verl}, Theorem~\ref{deltaprop}, and Corollary~\ref{1.9}
}\lbl{dim}

\subsection{Recursion formulas for 
 $\fo_g^{(2c)}$ and $\fe_g^{(2c)}$.}
\lbl{dim.1} Recall that 
$$ \fo_g^{(2c)} = \dim_{\BF_p}(F^{odd}(\Si))\text{\ \ \ and \ \ } 
\fe_g^{(2c)} = \dim_{\BF_p}(F(\Si)\slash F^{odd}(\Si))$$
(where $\Si=\Si_g(2c)$)   
are given by the number of odd or even colorings of the lollipop
tree  
$T_g^{(2c)}$ 
shown in Figure \ref{lol}.  
In  
a coloring of a lollipop tree, all the colors except the
loop colors must be even numbers $2a$ with $0\leq a\leq d-1$. 
Recall \cite{GM} that a stick edge is an edge which meets a loop. 
We will call
a trivalent vertex at the opposite end of a stick edge from the loop a stick vertex.  
 We say a coloring is {\em balanced} at a 
stick vertex 
 if the three edges which meet at the stick vertex 
 are colored 
 $2a$, $2b$, and $2c$ 
with 
$a+b+c \equiv 0 \pmod{2}.$  
Otherwise 
the coloring is called {\em unbalanced} at the stick vertex.  
 It is
easy to see 
from Definition~\ref{eoc} 
that a coloring  
is even if and only if the number of stick vertices where the coloring is unbalanced is even. 

As a building block, we consider first the lollipop tree
$T_{1}^{(2c_1,2c_2)}$ associated to a surface of genus one with two colored
points colored $2c_1$ and $2c_2$
shown in Figure~\ref{ltg1}.
Note  
that colorings of $T_{1}^{(0,2c)}$ are the same as colorings of
$T_{1}^{(2c)}$, 
as an arc which is colored zero can be erased. 
Let $\be(c_1,c_2)$ denote the number of colorings of $T_{1}^{(2c_1,2c_2)}$ that are balanced at the stick vertex and $\up(c_1,c_2)$ denote the number of colorings of $T_{1}^{(2c_1,2c_2)}$ that are unbalanced  at the stick vertex.

\begin{figure}[h]
\includegraphics[width=1.5in]{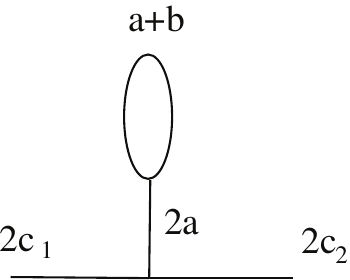}  
\caption{}\lbl{ltg1}
\end{figure}

\begin{lem}\lbl{lem71} If $c_1 \ge c_2$, then $\be(c_1,c_2)= (c_2+1)(d-c_1)$, and 
$\up(c_1,c_2)= c_2(d-c_1).$ 
\end{lem}

\begin{proof} Small admissibility at the top of the stick edge says simply that $0 \le b \le d-a -1.$
Admissibility at the bottom of the stick edge says 
\begin{equation*} c_1-c_2 \le a \le 
\min 
\{c_1+c_2,2d-1-c_1-c_2\}~. 
\end{equation*}

Thus  $\be(c_1,c_2)$ is either   
$$\sum_
{ 
\genfrac {}{}{0pt}{}{c_1 - c_2 \le  a \le c_1 + c_2}
{ a\equiv c_1-c_2 \pmod 2 }
} 
  (d-a) \text{\ \ \ \ \ \ or \ \ \ } \sum_
{
\genfrac {}{}{0pt}{}{c_1 - c_2 \le
   a \le 2d-1 -c_1 -c_2 }
  {a\equiv c_1-c_2 \pmod 2 }
  }
(d-a) 
$$ 
depending on  which upper limit for $a$ is smaller. But it turns
out that both these sums sum to $(c_2+1)(d-c_1)$.
The unbalanced case is done similarly.
\end{proof}

\begin{prop} The numbers $\fe_g^{(2c)}$ and $\fo_g^{(2c)}$  satisfy
  the following recursion formulas:
\begin{align} \lbl{f16}\fe_{1}^{(2c)}&=\delta_1^{(2c)} 
=D_{1}^{(2c)}
= d-c \text{ \  and\ } 
    \fo_{1}^{(2c)}=0\\
\lbl{f-ev} \fe_{g+1}^{(2c)}&= \sum_{a=0}^{d-1}(\fe_{g}^{(2a)}\be(a,c) +
\fo_{g}^{(2a)}\up(a,c))\\
\lbl{f-odd}\fo_{g+1}^{(2c)}&= \sum_{a=0}^{d-1}(\fo_{g}^{(2a)}\be(a,c) +
\fe_{g}^{(2a)}\up(a,c))
\end{align}
\end{prop}

\begin{proof}This follows from 
 the discussion preceding Lemma~\ref{lem71}.
\end{proof}

Of course the numbers $D_g^{(2c)}= \fe_g^{(2c)} +
\fo_g^{(2c)}$ and $\delta_g^{(2c)}= \fe_g^{(2c)} -
\fo_g^{(2c)}$ can also be computed recursively from
(\ref{f16})--(\ref{f-odd}). For later use, we remark that
\begin{equation} \lbl{recur1}
\delta_{g+1}^{(2c)}= \sum_{a=0}^{d-1}(\be(a,c)
-\up(a,c))\,\delta_{g}^{(2a)}= \sum_{a=0}^{d-1} 
(d- \max\{a,c\})\,\delta_{g}^{(2a)}~. 
\end{equation} 

It is not hard to get explicit formulas  for low genus in this way, 
such as:
\begin{align} 
\lbl{fex}\fe_2^{(2c)}&= 
 ( (c+1) p^3-(3c^2+3c) p^2+ (2 c^3-3 c-1)  p + 2 c^3+3 c^2+c )/24 \\
\lbl{fox}\fo_2^{(2c)}&= 
      ( c p^3 -(3c^2+3c) p^2+ (2 c^3+6c^2+3c)p -(2c^3+3c^2+c) )/24 \\
\lbl{fo}
\delta_2^{(2c)}&=  
       ( p^3 - (6c^2 +6c+1)p       + 4 c^3 +6 c^2 +2 c)/24\\
 \fo_3 &= {(p-3) (p-2) (p-1)^2 p (p+1)}/{2880}\\
\fe_3 &= (p-1) p (p+1)^2 (p+2) (p+3)/2880 \lbl{eo}\\
\delta_3&=  (p-1) p (p+1) (p^2+1)/240\lbl{f23}\\
\delta_4&= (p - 1) p (p + 1) (17 p^4  + 31 p^2  + 24)/40320\lbl{f24} \\
\delta_5&= (p - 1) p (p + 1) (31 p^6 + 82 p^4 + 103 p^2  + 72)/725760
\lbl{f25} \ .
\end{align}

\subsection{Proof of Theorem~\ref{new-Verl}.}  
\lbl{dim.2}
We must prove  
formula  (\ref{d-Verl}) for $\delta_g^{(2c)}$. 
To this end, 
let us reformulate the recursion relation
(\ref{recur1}) in terms of the
Verlinde algebra, as follows. The Verlinde algebra $V_p$ for the
$\SO(3)$-TQFT at the  
$p$-th
root of unity (where $p=2d+1$)  
\cite{BHMV1,BHMV2} is 
the quotient algebra   
$$V_p=K[z]/(e_{d}-e_{d-1})$$ where $K=\BQ(\zeta_p)$ is the cyclotomic field
and 
where $e_i=e_{i}(z)$ is  
the
Chebyshev polynomial of
the second kind 
defined recursively by $e_0=1, e_1(z)=z$, and
$e_{n+1}(z)=ze_n(z) -e_{n-1}(z)$. In $V_p$, one has
\begin{equation}\lbl{even-odd} e_{d+i}=e_{d-1-i}~,\end{equation} so
that the linear basis $\{e_0, e_1, \ldots,
e_{d-1}\}$ of $V_p$ is the same, up to reordering, as the basis $\{e_0, e_2, \ldots,
e_{2d-2}\}$. If $x\in V_p$, we denote by $M(x)=(M(x)_{ji})$ the
$d\times d$ matrix describing 
multiplication by $x$ in the basis $\{e_0, e_2, \ldots,
e_{2d-2}\}$. Here, we index the matrix entries from $0$ to $d-1$ so that
$$x\, e_{2i}=\sum_{j=0}^{d-1} M(x)_{ji}e_{2j}~.$$ 

The following result identifies
$(-1)^c  \delta_g^{(2c)}$ as a matrix coefficient.

\begin{prop}\lbl{6.7}  Define   
$\hKK\in V_p$ by
$$\hKK=
\sum_{n=0}^{d-1}(-1)^n (d-n)e_{2n}~.$$ Then for $0\leq c\leq d-1$, we have $$(-1)^c \delta_g^{(2c)}=M(\hKK^g)_{c,0}~.$$
\end{prop}
\begin{proof} Since $e_0$ is the identity element of the algebra
  $V_p$, we have $\hKK^ge_0=\hKK^g$ and it is enough to show that
  \begin{equation}\lbl{hK} \hKK^g=\sum_{c=0}^{d-1}(-1)^c \delta_g^{(2c)}
  e_{2c}~.
\end{equation}
We prove (\ref{hK}) by induction on $g$. The formula certainly holds for  $g=0, 1$ since
 $\delta_0^{(0)}=1$, $\delta_0^{(2c)}=0$ ($c \ne 0$) and 
  $\delta_1^{(2c)}=\fe_1^{(2c)}=d-c$. The induction step follows from
  our recursion formula (\ref{recur1}), since the matrix coefficients of $M(\hKK)$ are given
by $$M(\hKK)_{ji}=(-1)^{i+j}(d-\max\{i,j\})~, $$ 
as is easily checked using the fact that the structure
constants of the Verlinde algebra in the basis $\{e_0, e_2, \ldots,
e_{2d-2}\}$ encode the admissibility conditions at the trivalent vertices of a
colored graph. \end{proof}

The 
original Verlinde numbers $D_g^{(2c)}$ can also be described as matrix
coefficients:  put 
$$ \KK=\sum_{n=0}^{d-1}(d-n)e_{2n} = \sum_{n=0}^{d-1} e_{2n}^2~,$$ 
then 
\begin{equation}\lbl{kkf} 
D_g^{(2c)}=M(\KK^g)_{c,0}~.
\end{equation}
Moreover, it is well-known how to diagonalize the matrix $M(\KK)$
using the so-called $S$-matrix, and this
diagonalization  gives a proof of the Verlinde formula
(\ref{Verl}) (see Remark~\ref{PrV} below). 
We now apply the same method to the matrix $M(\hKK)$. 

Here is a formulation of the basic
diagonalization result. Let us write $q=\zeta_p$. Let $S=(S_{ij})$ be the
$d\times d$ matrix
  defined by $$S_{ij}=q^{(2i+1)(2j+1)}-q^{-(2i+1)(2j+1)} \ \ \ \ \ \ \ 
   (i,j=0,1, \ldots, d-1)~.$$ This matrix is $q-q^{-1}$ times the $S$-matrix of
   the $SO(3)$-TQFT, and one has
\begin{equation*}\lbl{Sm} S^{-1}=- \textstyle {\frac{1}{p}} S
\end{equation*}  (this last statement can also be checked by
   direct computation). 

\begin{lem} Let $Q=\diag(-q^{2j+1}-q^{-2j-1})_{j=0,1,\ldots,d-1}$. Then 
\begin{equation}\lbl{MZ} 
M(z)= S Q S^{-1}= - \textstyle {\frac{1}{p}} S Q S 
\end{equation} 
\end{lem}

\begin{proof}  
This statement can also be checked by a direct computation ({\em cf.}
\cite{GMGow}). 
Here is how to deduce it
  from results in \cite{BHMV2}. On page 913 of that paper, one finds elements
  $v_j\in V_p$ which are eigenvectors for the multiplication by $z$: 
   $$zv_j=-(q^{j+1}+q^{-j-1})v_j~.$$ It is observed there that $v_0,
  v_1, \ldots v_{d-1}$ form a basis of $V_p$, but using
  (\ref{even-odd}), one sees that this is the same, up to reordering,
  as $v_0,
  v_2, \ldots v_{2d-2}$, which is therefore also a basis.  The matrix $Q$ is the matrix of
  multiplication by $z$ in this latter basis. Moreover, using
  (\ref{even-odd}) again, one can check that in our notation  $v_{2j}$
  is given by
$$ v_{2j}= \frac{1}{q-q^{-1}}\sum_{i=0}^{d-1} S_{ij} e_{2i}~.$$ This proves the lemma.
\end{proof}

Let $G_j$ denote the Galois automorphism of the cyclotomic field
$Q(q)$ 
(recall $q=\zeta_p$ is a primitive $p$-th root of unity)
defined by $G_j(q)=q^j$ ($j=1,\ldots,p-1$). Thinking of  $\hKK$ as a
polynomial in $z$, we see from (\ref{MZ}) that the eigenvalues  of
$M(\hKK)$ are given by $G_{2j+1}(\Lambda)$ ($j=0,\ldots,d-1$), where
\[ 
\Lambda = \hKK(z)\Big\vert_{z=-q-q^{-1}}=\sum_{n=0}^{d-1}(-1)^n (d-n)
  \frac{q^{2n+1}-q^{-2n-1}}{q-q^{-1}}=
  \frac 1 {(q+q^{-1})^2}.
  \]
(To see the last equality, multiply both sides by
$(q-q^{-1})(q+q^{-1})^2$ and compare the
coefficients of the resulting polynomials.)
 
Thus (\ref{MZ}) gives 
\begin{align*}\notag M(\hKK^g)_{c,0}&=-
  \frac{1}{p} \sum_{j=0}^{d-1} S_{cj}G_{2j+1}(\Lambda^g)S_{j0}\\
&=  -
  \frac{1}{p} \sum_{j=0}^{d-1} G_{2j+1}
  \left((q^{2c+1}-q^{-2c-1})(q-q^{-1})\Lambda^g\right)\\
&=  -
  \frac{1}{p} \sum_{j=1}^{d} G_{j} \left((q^{2c+1}-q^{-2c-1})(q-q^{-1})\Lambda^g\right)
\end{align*}
The last equality is justified by the fact that both
$(q^{2c+1}-q^{-2c-1})(q-q^{-1})$ and $\Lambda$ lie in the real
subfield $\BQ(q+q^{-1})$ of the cyclotomic field $\BQ(q)$. 

Substituting now $q=-e^{\pi i/p}$ (which is a primitive $p$-th root of
unity), 
 and using Proposition~\ref{6.7}, 
we get the expression  (\ref{d-Verl}) 
for $\delta_g^{(2c)}$ 
in Theorem~\ref{new-Verl}.

\begin{rem}{\em In \cite[Eq.~(5)]{GMGow}, we wrote down a different
    formula
     for $\delta_g^{(2c)}$ which was based on the following
    expression for $ \Lambda $ as a polynomial in $q=\zeta_p$:
$$ \Lambda = \lceil  d/2 \rceil + \displaystyle{\sum_{k=1}^{d-1}} (-1)^k  \lceil
(d-k)/2 \rceil (q^{2k}+q^{-2k})~.$$ 
We thank Don Zagier for 
pointing out
 that this is equal to $1/(q+q^{-1})^2$, which leads to the simpler
formula for $\delta_g^{(2c)}$ given in Theorem~\ref{new-Verl}.}\end{rem}

\begin{rem}\lbl{PrV}{\em In the same way, one gets a proof of the
    Verlinde Formula (\ref{Verl}) by expressing $D_g^{(2c)}$ as a
    coefficient of the matrix  $M(\KK)$ as in (\ref{kkf}) and computing
    the eigenvalues of this matrix. The computation is exactly the same,
    except that $\Lambda$ must be replaced with
    $$\KK(z)\Big\vert_{z=-q-q^{-1}}=\frac {-p}{(q-q^{-1})^2}$$ (see
    \cite[p.~913]{BHMV2}). 
}\end{rem}

\subsection{Proof of Theorem~\ref{deltaprop}.} 
\lbl{dim.3}
Observe 
that
$\delta_1^{(2c)}=d-c$ 
is a polynomial in $p=2d+1$ and $c$ of total degree one, and 
\begin{equation}
  \lbl{recur}\delta_{g+1}^{(2c)}=\sum_{a=0}^{d-1}(d-a)\,\delta_{g}^{(2a)} 
+\sum_{a=0}^{
c-1}(a-c)\,\delta_{g}^{(2a)}~, 
\end{equation} 
 as follows easily from (\ref{recur1}).
Using the well-known formula 
\begin{equation} \lbl{Faul}\sum_{n=0}^{N-1} n^k= \frac{N^{k+1}}{k+1} +
  O(N^k)
\end{equation} (where $O(N^k)$
is a polynomial in $N$ of degree at most $k$), one 
sees
by
induction on $g$ that 
$\delta_g^{(2c)}$ is a 
polynomial of total degree at most $2g-1$ in $p$ and $c$. Let $L_g$
the homogeneous part of degree $2g-1$ in $\delta_g^{(2c)}$.
We must show that

\begin{equation}\lbl{LG1} 
L_g=(-1)^{g-1} \sum_{k=1}^{2g} \,2(2^k-1)\, \frac{B_k}{k!}\,
\frac{c^{2g-k}}{(2g-k)!}\, p^{k-1}~.
\end{equation}

This will also be proved by induction. It is certainly
true for $g=1$. To perform the induction step, assume it is true for
$g$. Using (\ref{recur}) and (\ref{Faul}), it is easy to see
that a monomial
$c^{2g-k}p^{k-1}$ in $L_g$ gives rise to a term 
\begin{equation}\lbl{rrr} \frac{-c^{2g+2-k}p^{k-1} + 2^{k-2g-2} p^{2g+1}} {(2g-k+1)(2g-k+2)}\end{equation}
in $L_{g+1}$. We must show that the coefficient of $c^{2g+2-k}p^{k-1}$
in $L_{g+1}$ is given by (\ref{LG1}) with
$g+1$ in place of $g$. This is immediate  from (\ref{rrr}) for
$k<2g+2$, while for $k=2g+2$ it follows from the identity
\begin{equation}\lbl{bern} \sum_{k=0}^{2g+2}2^k(2^k-1) \binom{2g+2}{k}
  B_k = 0~. \end{equation} The identity (\ref{bern}) follows from
Exercises 12, 17, 19, 21, and 22, in
\cite[Ch.~15]{IR}.\footnote{Warning: there is a factor $2^{n-1}$
  (resp. $a^{n-1}$) missing in the 
  statements of Exercises 21 (resp. 22).}

\subsection{A residue formula for $D_g^{(2c)}$.} 
\lbl{dim.4}
For the proof of Corollary~\ref{1.9}, we need the following expression for the Verlinde
formula (\ref{Verl}) which is computed using  the residue theorem.  This computation is
well-known in the case $c=0$ (see 
the references in \cite{BHMV2}.) 
 If $c\neq 0$, however, an
additional binomial coefficient appears in the formula, which we have
not seen in the literature. Therefore we sketch the computation
here. We use the notation  
\begin{equation}\lbl{sinh}
\mathrm{s}(t)=\frac{\sinh(t)}{t}= \sum_{k=0}^\infty \frac {t^{2k}}{(2k+1)!}~.
\end{equation}
\begin{prop}\lbl{6.6} (Residue formula for $D_g^{(2c)}$.)  For 
$g\geq 1$,
one has
\begin{equation}\lbl{f42}
D_g^{(2c)} = \frac{(-p)^g}{2}  
\Bigg(
4^{1-g}\frac{2c+1}{p}   \res_{t=0}\left(\frac{2pt}{e^{2pt}-1} 
\frac{\mathrm{s}\left((2c+1)t\right)}{\mathrm{s}(t)^{2g-1}}
 \frac{dt}{t^{2g-1}}\right)
- \binom{c+g-1}{2g-2}
\Bigg).
\end{equation}
\end{prop}

\begin{proof} ({\em Cf.} \cite[p.~914]{BHMV2}.) We apply the residue theorem to the
  meromorphic $1$-form $$\frac{(z^{2c+1}-z^{-2c-1})dz}{z(z^p-1)(z-z^{-1})^{2g-1}}~.$$
  This form has simple poles at all non-trivial $p$-th roots
  of unity, and the sum of the residues at these poles is 
  $-2D_g^{(2c)}/(-p)^g$
   by the Verlinde formula (\ref{Verl}). The poles
  at $z=\pm 1$ give rise to the residue term in (\ref{f42}), and the
  pole at $z=0$ gives rise to the binomial coefficient in
  (\ref{f42}). (This term would be absent in the case $c=0$.) There are no other poles. Thus the result follows
  from the residue theorem.
\end{proof}

Using the power series expansions for
$t/(e^t-1)$ in (\ref{bernn}), and for 
$\mathrm{s}(t)$ in (\ref{sinh}), it is easy to see from
(\ref{f42}) that for  $g\geq 2$,  $D_g^{(2c)}$ is a polynomial of total degree
$3g-2$ in $p$ and $c$. Moreover, it is easy to get an explicit formula
for the leading order terms. One finds
that for 
$g\geq 2$, 
 and in degrees $\geq 3g-3$, the polynomial $D_g^{(2c)}/2$ is given by the
expression appearing on the right hand side of Eq.~(\ref{1.99}) in Corollary~\ref{1.9}.
 
Corollary~\ref{1.9} follows from this by observing that both $\fe_g^{(2c)}$ and
  $\fo_g^{(2c)}$ coincide with $D_g^{(2c)}/2$ up to addition  of some polynomial
  of degree $2g-1$, by Theorem~\ref{deltaprop}.

\subsection{A residue formula for  $\delta_g^{(2c)}$.} 
\lbl{7.5}

Here is
an analog of Proposition \ref{6.6} for $\delta_g^{(2c)}$.

\begin{prop}\lbl{nres}  
 (Residue formula for $\delta_g^{(2c)}$.) 
For $g\geq 1$, one has

\begin{align} \lbl{nresid}
 \delta_g^{(2c)}=  \frac{(-1)^g}{2} 
\Bigg(
 \frac{4^{1-g}}{p}  \res_{t=0} &
\left(
\frac{2pt}{e^{2pt}+1} 
\frac{\cosh((2c+1)t) \cosh (t)} {\mathrm{s}(t)^{2g}} 
\frac{dt}{t^{2g+1}} 
\right) \\
& +\frac{2c+1}{2g-1} \binom{c+g-1}{2g-2}  
\Bigg). \notag
\end{align}

\end{prop}

\begin{proof}  Consider the
  meromorphic $1$-form $$\frac{(z^{2c+1}-z^{-2c-1})(z-z^{-1})dz}{z(z^p-1)(z+z^{-1})^{2g}}~.$$
  This form has simple poles at all non-trivial $p$-th roots
  of unity, and the sum of the residues at these poles is 
  $(-1)^{c+1} 2\delta_g^{(2c)}$
   by   
formula (\ref{d-Verl}). The poles
  at $z=\pm i$ give rise to the residue term in (\ref{nresid}), and the
  pole at $z=0$ gives rise to the  
binomial coefficient term 
in
  (\ref{nresid}).
    There are no other poles. Thus the result follows, as before,
  from the residue theorem.
\end{proof}

\begin{rem}{\em Formula (\ref{nresid}) shows again that $\delta_g^{(2c)}$
  is a polynomial in $c$ and $p$ of total degree
  $2g-1$. Using 
\begin{equation}\lbl{formula42}
\frac t {e^t+1}= \sum_{n=1}^\infty (1-2^n)B_n
\frac {t^n} {n!}
\end{equation}
(as follows easily from (\ref{bernn})), one can write down an explicit
formula for this polynomial in terms of Bernoulli numbers. This could
be used to give another proof of Theorem~\ref{deltaprop}. Below, we
make two more 
remarks about this polynomial. 

First, the contribution to $\delta_g^{(2c)}$ coming from the residue
term in (\ref{nresid}) is of the form $p$ times a polynomial in
$(2c+1)^2$
 and $p^2$. (This is because both
  $\cosh (t)$ and $\mathrm{s}(t)$ are even functions of $t$; note also
  that there
  is no constant term in (\ref{formula42}).) Thus, any monomial  $c^n p^m$  
      appearing  
    in $\delta_{g}^{(2c)}$ has $m$ odd or
    zero. (See for example Formula (\ref{fo}).)

Second, the contribution to $\delta_g^{(2c)}$ coming from the binomial
coefficient is zero for $c<g-1$. In particular, for $g\geq 2$, the
one-variable polynomial $\delta_{g}$ obtained by putting $c=0$ in $\delta_{g}^{(2c)}$ is
an odd polynomial in $p$. We claim that for $g\geq 2$ this polynomial
$\delta_{g}$ is always divisible by 
 $$\delta_2=\fe_2=D_2= p(p^2-1)/24~.$$ (See for example Formulas  
(\ref{f23})---(\ref{f25})). This can be seen as
 follows. First, divisibility by $p$ is clear since $\delta_{g}$ is
an odd polynomial in $p$.  
Next, running the
residue computation in the proof of Proposition~\ref{nres} backwards
in the special case $c=0$, $p=1$, one finds that
the right hand side of (\ref{nresid}) is zero in this case. Thus the
polynomial $\delta_{g}$ is divisible by $p-1$. But since
$\delta_{g}$ is odd, it must then also be divisible by $p+1$. This
proves the claim.

As mentioned in the introduction, we 
believe that the polynomial $\delta_{g}^{(2c)}$ and its specialization
$\delta_{g}$ should have an algebro-geometric interpretation. If so,
there should probably be a geometric reason behind the above-mentioned
properties of these polynomials.
}\end{rem}

\section{Further Comments}\lbl{sec7}

For $g\geq p-1$,  Gow \cite{Go} has constructed  $p-1$ irreducible
modular $K$-representations of the symplectic group $\Sp(2g,K)$, where
$K$ is a field of characteristic $p$. Gow denotes these representations by 
\begin{equation*}\qquad \qquad \qquad \quad V(g,k) \quad \quad (g-p+2
  \leq k \leq g)~.
\end{equation*}
 The representation $V(g,k)$ is a subquotient of $\Lambda^kV$ where
$V\simeq K^{2g}$
is the standard representation of $\Sp(2g,K)$. If $K$ is algebraically
closed, then  $V(g,k)$ is the fundamental module with highest weight
$\omega_k$ \cite[Corollary 2.4]{Go}. 

The dimensions of these representations of $\Sp(2g,K)$ depend only on the
characteristic 
 $p$ of 
the field $K$.
 It turns out that for $p=5$, the dimensions of the
four representations constructed by Gow coincide with the dimensions
of our irreducible factors $F^{odd}(\Si_g(2c))$ and
$F(\Si_g(2c))\slash F^{odd}(\Si_g(2c))$:

\begin{thm}\label{gw}  For $p=5$ and $g\geq 4$, we have $$(\dim
  V(g,g+1-n))_{n=1,2,3,4} = (
  \fe_g^{(0)}(5),\fe_g^{(2)}(5),\fo_g^{(2)}(5),\fo_g^{(0)}(5))~.$$
\end{thm} 
  This is proved in \cite{GMGow}. It would be interesting to know
  whether these equalities of dimensions 
come from isomorphisms of the corresponding
$\Sp(2g,\BF_p)$-representa\-tions. 
In \cite{GMGow}, we give an explicit formula for the dimensions of the
$V(g,k)$ analogous to formulas (\ref{d-Verl}) and (\ref{Verl}) for
$\delta_g^{(2c)}$ and $D_g^{(2c)}$.  
We remark that for 
every prime $p\geq 5$  
we have {\em a priori} as many irreducible representations of $\Sp(2g,\BF_p)$ as
Gow.\footnote{{\em A priori,} we have $p-1$ irreducible representations of
  $\Sp(2g,\BF_p)$, since we have two of them for each $0\leq c\leq
(p-3)/2$). For $g\geq 3$, these representations are all
non-trivial, but in principle some of them could be isomorphic,
although we don't expect this to happen.} But it appears that 
for $p>5$  the dimensions  of the $V(g,k)$ and the dimensions of our  irreducible
factors are different.


\begin{thebibliography}{BHMVZ}



\bibitem[A]{A} {\sc J. E. Andersen.} Mapping Class Groups do not have
  Kazhdan's Property (T). Preprint 2007 {\tt arXiv:0706.2184}

 

\bibitem[AM]{AM} {\sc J. E. Andersen, G. Masbaum.} Involutions on
  Moduli Spaces and Refinements of the Verlinde Formula. {\em
    Math. Annalen}  {\bf 314}  (1999) 291-326. 


\bibitem[BHMV1] {BHMV1} {\sc C. Blanchet, N. Habegger, G. Masbaum,
P. Vogel.} Three-manifold invariants derived from the Kauffman bracket.
{\em Topology} {\bf 31} (1992), 685-699.


\bibitem[BHMV2]{BHMV2} {\sc  C. ~Blanchet, N. ~Habegger, G. ~Masbaum,  P. ~Vogel.}
Topological quantum field theories derived from the Kauffman bracket,  {\em Topology}   {\bf 34} (1995), 883-927







\bibitem[Gi]{G} {\sc P. M. Gilmer.}  {Integrality for TQFTs}, {\em Duke Math J}, {\bf 125} (2004), no. 2, 389--413  




\bibitem[GM1]{GM} {\sc P. M. Gilmer, G. Masbaum.} Integral lattices in 
TQFT. { \em Annales Scientifiques de l'Ecole Normale Superieure}, {\bf
  40}, (2007), 815--844

\bibitem[GM2]{GM2}{\sc P. M. Gilmer, G. Masbaum.} 
{Integral TQFT for a one-holed torus, }
 {\em Pacific J. Math.} {\bf 252}, (2011), 93--112

\bibitem[GM3]{GM3}{\sc P. M. Gilmer, G. Masbaum.} 
{Maslov index, Lagrangians, Mapping Class Groups and TQFT}, arXiv:0912.4706,
to appear in {\em Forum Mathematicum.}


\bibitem[GM4]{GMGow}{\sc P. M. Gilmer, G. Masbaum.} A
  dimension formula for Gow's modular representations of the
  symplectic group in the natural characteristic, { \em J. Pure and Appl. Algebra}, {\bf 217} (2013), 82--86
arXiv:1111.0240

\bibitem[GMW]{GMW} { \sc P. M. Gilmer, G. Masbaum, P. van Wamelen.}
  Integral bases for TQFT modules and unimodular representations of
  mapping class groups. { \em Comment. Math. Helv.} {\bf 79} (2004),
  260--284.

  \bibitem[Go]{Go} {\sc R.~Gow.}  Construction of $p-1$ irreducible modules with fundamental highest 
weight for the symplectic group in characteristic $p$, { \em J. London Math. Soc.} 
{\bf58} (1998) 619--632.
 
  
\bibitem[IR]{IR} {\sc K. Ireland, M. Rosen.} A classical introduction
  to modern number theory. Graduate texts in mathematics {\bf 84},
  Springer Verlag, Second Edition 1990


\bibitem[K]{Ka} {\sc L. H.~Kauffman.} 
State models and the Jones polynomial.
{ \em Topology} {\bf 26} (1987),  395Ð-407. 








\bibitem[M]{M}  {\sc G. Masbaum.}   
(In preparation.)



\bibitem[MRe]{MRe}  {\sc G. Masbaum, A. W. Reid.}  All finite groups are involved in the Mapping Class Group.  { \em Geometry and Topology} {\bf 16} (2012), 1393-1411


 


      \bibitem[MRo]{MR2}
{\sc G. Masbaum, J. D. Roberts.} A simple proof of
integrality of quantum invariants at prime roots of unity. {\em
  Math. Proc. Cam. Phil. Soc.} {\bf 121} (1997), 443-454.


      \bibitem[Ro]{R}
{\sc J. D.  Roberts.} Irreducibility of some quantum representations of mapping class groups. 
{ \em J. Knot Theory Ramifications} {\bf 10}, 763-767 (2001). 

\bibitem[T]{T} {\sc V. Turaev.} Quantum invariants of knots and 3-manifolds. De
Gruyter Studies in mathematics 18, 1994


\bibitem[Wa]{W}
{ \sc K. ~Walker.} {On Witten's 3-manifold invariants}, Preliminary
Version, 1991  \\ \url{http://canyon23.net/math/} 

\bibitem[Wi]{Wi}{ \sc E. Witten.} On quantum gauge theories in two
  dimensions. {\em Comm. Math. Phys.} {\bf 141}, 153-209 (1991)

\bibitem[Z]{Z}
{ \sc D.~Zagier.} Elementary aspects of the Verlinde formula and the
Harder-Narasimhan-Atiyah-Bott formula. {\em Israel mathematical
  Conference Proceedings} {\bf 9} 445-462 (1996)




\end{thebibliography}
\end{document}